\documentclass[11pt]{amsart}
\usepackage[margin=1.5cm]{geometry}

\usepackage{amsmath,amsfonts,amssymb,amsthm}
\usepackage[all,cmtip]{xy}
\usepackage[
backend = biber,
citestyle = alphabetic,
style = alphabetic,
maxnames = 50,
giveninits=true,
]{biblatex}
\usepackage[hidelinks]{hyperref}
\usepackage{mathtools}
\usepackage[T2A,T1]{fontenc}
\DeclareSymbolFont{cyrillic}{T2A}{cmr}{m}{n}
\DeclareMathSymbol{\Sha}{\mathalpha}{cyrillic}{216}
\usepackage{xcolor}
\usepackage[normalem]{ulem}

\numberwithin{equation}{section}

\newtheorem{dummy}{dummy}[section]

\newtheorem{definition}[dummy]{Definition}
\newtheorem{theorem}[dummy]{Theorem}
\newtheorem{corollary}[dummy]{Corollary}
\newtheorem{lemma}[dummy]{Lemma}
\newtheorem{proposition}[dummy]{Proposition}
\newtheorem{remark}[dummy]{Remark}
\newtheorem{example}[dummy]{Example}
\newtheorem{question}[dummy]{Question}

\newcommand{\Z}{\ensuremath{\mathbb{Z}}}
\newcommand{\F}{\ensuremath{\mathbb{F}}}
\newcommand{\Q}{\ensuremath{\mathbb{Q}}}

\newcommand{\CC}{\ensuremath{\mathbb{C}}}
\newcommand{\PP}{\ensuremath{\mathbb{P}}}

\newcommand{\DD}{\ensuremath{\mathcal{D}}}

\newcommand{\sheaf}[1]{\ensuremath{\mathcal{#1}}}

\def\Db{\DD^{b}}

\def\NS{\operatorname{NS}}
\def\Pic{\operatorname{Pic}}
\def\Aut{\operatorname{Aut}}
\def\Br{\operatorname{Br}}
\def\Image{\operatorname{Im}}
\def\Ker{\operatorname{Ker}}
\def\rk{\operatorname{rk}}
\def\ind{\operatorname{ind}}
\def\ord{\operatorname{ord}}

\def\DE{\operatorname{DE}}
\def\HE{\operatorname{HE}}
\def\Lagr{\operatorname{L}}
\def\tLagr{\operatorname{\widetilde{L}}}
\def\Jac{\operatorname{J}}
\def\II{\operatorname{I}}
\def\tII{\widetilde{\operatorname{I}}}

\def\EllK3{{\operatorname{EllK3}}}
\def\BrK3{{\operatorname{BrK3}}}

\def\ol{\overline}
\def\wt{\widetilde}

\title{Derived equivalence for elliptic K3 surfaces and Jacobians}

\author[R.~Meinsma]{Reinder Meinsma}
\address{
School of Mathematics and Statistics, University of Sheffield,
Hounsfield Road, S3 7RH, UK, and
Hausdorff Center for Mathematics
at the University of Bonn, Endenicher Allee 60, 53115.
}
\email{r.r.meinsma@gmail.com}
\author[E.~Shinder]{Evgeny Shinder}
\address{
School of Mathematics and Statistics, University of Sheffield,
Hounsfield Road, S3 7RH, UK, and
Hausdorff Center for Mathematics
at the University of Bonn, Endenicher Allee 60, 53115.
}
\email{eugene.shinder@gmail.com}
\thanks{E.S. was partially supported by the EPSRC grant
    EP/T019379/1 ``Derived categories and algebraic K-theory of singularities'', and by the
    ERC Synergy grant ``Modern Aspects of Geometry: Categories, Cycles and Cohomology of Hyperkähler Varieties".}

\begin{document}
\begin{abstract}
We present a 
detailed study of elliptic fibrations on Fourier-Mukai partners of K3 surfaces, which we call derived elliptic structures. 
We fully classify derived elliptic structures
    in terms of Hodge-theoretic data, similar to the Derived Torelli Theorem that describes Fourier-Mukai partners. 
    In Picard rank two, derived elliptic structures are fully determined by the Lagrangian subgroups of the discriminant group.
    As a consequence, we prove that
    for a large class of Picard rank 2 elliptic K3 surfaces all Fourier-Mukai partners are Jacobians,
    and we partially extend this result to non-closed fields.
    We also show that there exist elliptic K3 surfaces with Fourier-Mukai partners which are not Jacobians of the original K3 surface. This gives a negative answer to a question raised by Hassett and Tschinkel.
\end{abstract}
\maketitle
 \tableofcontents

\section{Introduction}

Study of derived equivalence for complex
K3 surfaces
goes back to the work of Mukai.
By the Derived Torelli Theorem \cite{Muk87, Orl03},
derived equivalence translates
to a Hodge-theoretic
concept. 
Building on the Derived Torelli
theorem, 
and Nikulin's work on lattices \cite{Nik80},
one can deduce a formula for the number
of Fourier-Mukai partners for a complex K3 surface
\cite{Ogu01, HLOY02}.

Derived equivalences
of elliptic K3 surfaces have been studied in
\cite{Ste02, Gee05}.
One way to produce Fourier-Mukai partners
of an elliptic surface $f: X \to \PP^1$, is to take Jacobians $\Jac^k(X)$, which
are moduli spaces parametrising stable
torsion sheaves supported on a fibre of $f$
and having degree $k \in \Z$.
If $k$ is coprime to the multisection index of $f$,
then $\Jac^k(X)$ is derived equivalent to $X$ and we refer to $\Jac^k(X)$ as a \textit{coprime Jacobian} of $X$.
This raises the question of whether the converse is also true:

\begin{question}
\label{q:HT}
Is every Fourier-Mukai partner of an elliptic surface
$X$ a coprime Jacobian of $X$?
\end{question}

Question \ref{q:HT} was asked in 2014 by Hassett and Tschinkel in the case $X$ is a K3 surface \cite[Question 20]{HT14}. 
In fact, since elliptic K3 surfaces
can have several non-isomorphic elliptic fibrations,
one can interpret this question differently
depending on whether we fix a fibration on $X$ 
in advance or not.

For elliptic surfaces of non-zero Kodaira dimension, as well as for bielliptic and Enriques surfaces,
\cite{BM01, BM19}, Question \ref{q:HT}
has an affirmative answer.
We do not know the answer in the abelian case.

One of our main results is the following
answer to Question \ref{q:HT} for 
K3 surfaces:

\begin{theorem}[See Corollaries \ref{cor:HT-DJ} and \ref{cor:DE-count}]\label{thm:main-intro}
Let $X$ be an elliptic K3 surface of Picard rank 2. Let $t$ be the multisection index of $X$ and let $2d$ be the degree of a polarisation on $X$. Denote $m=\operatorname{gcd}(d,t)$. 

\begin{enumerate}
    \item[(i)] If $m=1$, then 
    every Fourier-Mukai partner of $X$
    is isomorphic to a coprime Jacobian
    of a fixed elliptic fibration on $X$;
    
    \item[(ii)] If $m=p^k$, for a prime $p$, then every Fourier-Mukai
    partner of $X$ is isomorphic
    to a coprime Jacobian
    of one of the two elliptic fibrations on $X$;
    
    \item[(iii)] If $m$ is not a power of a prime, and $X$ is very general with these properties, then $X$ admits Fourier-Mukai partners which are not isomorphic to any Jacobian of any elliptic fibration on $X$.
\end{enumerate}
\end{theorem}

Our method of proof of Theorem \ref{thm:main-intro}
relies on the Ogg-Shafarevich theory for elliptic surfaces, the Derived Torelli Theorem and lattice theory.
In addition we introduce a new ingredient: a \emph{derived elliptic structure}. 
The notion of the derived elliptic structure 
goes into the direction of describing
an elliptic structure on $X$ and its Fourier-Mukai partner
in terms
of the derived category $\Db(X)$.
We define a derived elliptic structure on a 
K3 surface $X$ as a choice
of an elliptic fibration on a Fourier-Mukai partner of $X$.
Using this language, Question \ref{q:HT} translates
to the question whether every derived elliptic structure on $X$
is isomorphic to a coprime
Jacobian of an actual elliptic structure on $X$.

We proceed to completely classify derived elliptic structures, for an
elliptic K3 surface $X$ of Picard rank two, in terms of certain \emph{Lagrangian subgroups} of the discriminant lattice $A_{\NS(X)}$
of the Neron-Severi lattice of $X$.
The final answer, at least when $X$ is very general, is that the number of derived elliptic structures on $X$, up to coprime Jacobians, 
equals 
$2^{\omega(m)}$
where $m$ is as in Theorem \ref{thm:main-intro}
and $\omega(m)$ is the number of distinct prime factors of $m$,
that is $\omega(1) = 0$,
$\omega(p^k) = 1$ and $\omega(m) > 1$ otherwise.
This explains
the condition on $m$
appearing in 
Theorem \ref{thm:main-intro}.

Let us
explain some difficulties that we encounter along the way. First of all, elliptic K3 surfaces of Picard rank two can have one or two elliptic fibrations, and in the latter case these elliptic fibrations are sometimes isomorphic. Thus, a direct comparison between the number of coprime Jacobians and Fourier-Mukai partners is complicated.

Secondly, many results that we state
for arbitrary elliptic K3 surfaces $X$ of Picard rank two simplify considerably when $X$ is very general. 
Indeed in this case, the group
$G_X$ of Hodge isometries of the transcendental lattice
$T(X)$ is trivial, that is $G_X = \{\pm 1\}$.
In general this is a finite cyclic group of even
order $|G_X| \le 66$.
This group appears in various bijections we produce, similarly
to how it appears in the counting formula of Fourier--Mukai partners
\cite{HLOY02}.
The set of isomorphism classes of
derived elliptic structures on $X$ is in natural bijection with the set
\[
\tLagr(A_{T(X)}) / G_X,
\]
see Theorem \ref{thm:derivedelliptic-lagrangianelts}.
Here $A_{T(X)}$ is the discriminant lattice of the transcendental lattice
$T(X)$,
and $\tLagr(A_{T(X)})$
denotes the set of Lagrangian elements (Definition \ref{def:lagrangianldt}).
Taking a coprime Jacobian $\Jac^k$ of an elliptic structure
translates into 
multiplying the corresponding Lagrangian element by $k$ and changing elliptic fibrations on a given surface corresponds to an 
involution which can be described intrinsically in terms of $A_{T(X)}$.
For very general $X$, $G_X = \{\pm 1\}$,
and this group acts by multiplying Lagrangian
elements by $-1$.
On the other hand, special $X$ will have fewer Fourier-Mukai partners and fewer coprime Jacobians, however they
will still match perfectly in cases
(i) and (ii) of Theorem \ref{thm:main-intro}.
See Example \ref{ex:thechosenone} 
for the
 most special
 (in terms of the size
 of $G_X$ and $\Aut(X)$)
elliptic K3 surface.

Similarly, when considering very general 
elliptic K3 surfaces, every isomorphism preserving the fibre
class is necessarily an isomorphism over the base.
This is false in general, and this is important, because
the Ogg-Shafarevich theory works with elliptic surfaces
over the base, whereas the natural equivalence relation
is that of preserving the elliptic pencil. We provide a careful analysis of the difference between isomorphism over $\PP^1$ and isomorphism as elliptic surfaces, which can
be of independent interest. In particular, we are able to state which of the coprime Jacobians
$\Jac^k(X)$ of an elliptic K3 surfaces
$X$ are isomorphic as elliptic surfaces (resp. over $\PP^1$).
Indeed, very 
general elliptic K3 surfaces
with multisection index $t$ have at most
$\frac{\phi(t)}{2}$ coprime
Jacobians, and the
explicit number can be
computed in all cases
as follows:
\begin{proposition}(see Proposition \ref{prop:jacobians-isom})
Let $X$ be a complex elliptic K3 surface.
There exist explicitly defined
cyclic subgroups $B_X \subset \wt{B}_X$ of $(\Z/t\Z)^*$,
such that
the number of isomorphism classes of coprime
Jacobians $\Jac^k(X)$ considered up to isomorphism
over the base (resp. preserving the elliptic pencil)
equals ${\phi(t)}/{|B_X|}$ 
(resp. ${\phi(t)}/{|\wt{B}_X|}$).
\end{proposition}

The group $B_X$ can only be non-trivial if $X$ is isotrivial with $j$-invariant $0$ or $1728$.
We give examples when $B_X$ and $\wt{B}_X$ are non-trivial, and when they are different.

\subsection*{Applications}
We deduce from Theorem \ref{thm:main-intro}
that zeroth 
Jacobians of derived equivalent elliptic K3 surfaces are non-isomorphic in general (Corollary \ref{cor:Jac0-different}), that is passing to the Jacobian 
can not be defined solely in terms of the derived category (Remark \ref{rem:DE-Jac0}).

Furthermore, Theorem \ref{thm:main-intro}
is relevant every time potential 
consequences
of derived equivalence between K3 surfaces
are considered.
Let us explain two non-trivial 
situations when the explicit or geometric
form of derived equivalence is desirable.
The first is rational points over non-closed fields and the second is L-equivalence.

The motivation of Hassett--Tschinkel \cite{HT14}
was the question of existence
of rational points
on derived equivalent elliptic K3 surfaces
over non-closed fields.
Namely, since $X$ and any of
its coprime Jacobians $\Jac^k(X)$
are isogenous, it follows that
$X$ has a rational 
point if and only if $\Jac^k(X)$ has a rational
point by the Lang-Nishimura theorem.
Using Galois descent, 
as we know automorphism groups
of elliptic K3 surfaces quite explicitly,
we can partially extend
Theorem \ref{thm:main-intro}
to subfields $k \subset \CC$,
and deduce the implication about rational points of Fourier-Mukai partners
(see Corollary \ref{cor:ratpointfmpartners}).
We note that the question about
the simultaneous existence
of rational points
on derived equivalent K3 surfaces still seems to be open.

Another application for Theorem \ref{thm:main-intro} is to the question of L-equivalence of derived equivalent
K3 surfaces $X$, $Y$
\cite{KuznetsovShinder}.
For elliptic K3 surfaces the natural strategy
is to prove L-equivalence for the generic fibres, which are genus one curves
over the function field of the base, and then spread-out the L-equivalence over the total space. This strategy has been realized in \cite{SZ20}
for elliptic K3 surfaces of multisection index five. It follows from Theorem \ref{thm:main-intro}
that the same approach can work 
when the mutlisection index $t$ is a power of a prime (and $d$ is arbitrary).

\subsection*{Structure of the paper} In Section \ref{sec:prelim},
we recall basic classical
results about lattices and complex K3 surfaces, and
moduli spaces of sheaves on K3 surfaces.
In Section \ref{sec:ellK3-rk2},
we describe in detail the elliptic K3 surfaces of rank two, including their Neron-Severi lattices, 
Lagrangian elements in their discriminant lattices, Hodge isometries of the transcendental lattices and the group of automorphisms. Most results in this section are standard except the focus on the Lagrangian elements.
In Section \ref{sec:Jacobians},
we recall the Ogg--Shafarevich theory and explain in detail when different Jacobians of a given elliptic fibration are isomorphic.
In Section \ref{sec:DEFM},
we introduce derived elliptic structures and Hodge elliptic structures on a K3 surface 
and fully classify them in terms of Lagrangian elements in the case of Picard rank two.  

\subsection*{Acknowledgements}

We thank Arend Bayer, Tom Bridgeland, Daniel Huybrechts,
Alexander Kuznetsov,
Gebhard Martin,
Giacomo Mezzedimi,
Sofia Tirabassi and
Mauro Varesco
for discussions and
interest in our work.

\section{Preliminary results}
\label{sec:prelim}

\subsection{Lattices} 
Our main reference for lattice theory is \cite{Nik80}. 
A lattice is a finitely generated free abelian group $L$ together with a symmetric non-degenerate bilinear form $b: L\times L\rightarrow \Z$. We consider
the quadratic
form $q(x) = b(x,x)$ and sometimes we write
$x \cdot y$ for $b(x,y)$ and $x^2$ for $q(x)$. 
A morphism of lattices between $(L,b)$ and $(L',b')$ is a group homomorphism $\sigma: L\rightarrow L'$ which respects the bilinear forms, meaning $b(x,y) = b'(\sigma(x),\sigma(y))$ for all $x,y\in L$. An isomorphism of lattices is called an isometry. We write $O(L)$ for the group of isometries of $L$. 
The lattice $L$ is called even if $x^2$ is even for all $x\in L$.
All the lattices we consider will be assumed to be even.

The dual of a lattice $L$ is defined as $L^*\coloneqq\operatorname{Hom}(L,\Z)$. It comes equipped with a natural bilinear form taking values in $\Q$. The bilinear
form gives rise to a natural map $L\to L^*$
which is injective because we assume $b$ to be non-degenerate;
furthermore we have a canonical isomorphism
\begin{equation}\label{eq:characterisingduallattice}
    L^* \simeq \left\{x\in L\otimes \Q\mid \forall y\in L: x\cdot y\in \Z \right\} \subseteq L\otimes \Q.
\end{equation}
The quotient $L^*/L = A_L$ is called the discriminant group of $L$. 
If the discriminant group is trivial, we call $L$ unimodular.
The discriminant group comes equipped with a quadratic form $\overline{q}: A_L\rightarrow \Q/2\Z$. 
There is
an orthogonal direct sum decomposition
\begin{equation}
\label{eq:A-primary}
A_L = \bigoplus_p A_L^{(p)}    
\end{equation}
where $A_L^{(p)}$
consists of elements annihilated by a power of a prime $p$. The group $A_L^{(p)}$ coincides
with the discriminant group of the $p$-adic lattice
$L \otimes \Z_p$.
Two lattices $L,L'$ are said to be in the same genus if they have the same signature and have isometric discriminant groups.  

An \textit{overlattice} of a lattice $T$ is a
lattice $L$ together with an embedding of lattices $T\hookrightarrow L$ of finite index. 
We say that two overlattices $T\hookrightarrow L$ and $T'\hookrightarrow L'$ are isomorphic if there exists a commutative diagram 
$$\xymatrix{
T \ar[r] \ar[d]_\sigma & L \ar[d]^\tau\\ T' \ar[r] & L'
}$$ where $\sigma$ and $\tau$ are isometries.

For any overlattice $T\hookrightarrow L$, there is a natural embedding of the cokernel $H_L\coloneqq L/T$ in the discriminant group of $T$ via the chain of embeddings $$T\hookrightarrow L\hookrightarrow L^*\hookrightarrow T^*.$$ The subgroup $H_L$ is isotropic with respect to the quadratic form on $A_T$,
and conversely any isotropic
subgroup of $A_T$ gives rise to an overlattice of $T$.
The following result gives a complete
classification of all overlattices of
a given lattice $T$, up to isomorphism.

\begin{lemma}[{\cite[Proposition 1.4.2]{Nik80}}] \label{lem:overlatsubgrcorresp}
Let $T$ be a lattice, and let $T\hookrightarrow L$ and $T\hookrightarrow M$ be two overlattices of $T$. An isometry $\sigma\in O(T)$ fits into a commutative diagram of the form 
\begin{equation}\label{diag:overlatticesoriginal}
    \xymatrix{
    T \ar[r]\ar[d]^-{\sigma} & L \ar[d]^-{\simeq} \\
    T \ar[r] & M
    }
\end{equation}
if and only if the induced isometry $\overline{\sigma}\in O(A_T)$ satisfies $\overline{\sigma}(H_L) = H_M$. Moreover, the assignment $(T\hookrightarrow L) \mapsto H_L$ is a bijection between the set of isomorphism classes of overlattices of $T$ and the set of $O(T)$-orbits of isotropic subgroups of $A_T.$
\end{lemma}

Note that \eqref{diag:overlatticesoriginal} can be completed as follows:

\begin{equation} \label{diag:overlatticesoriginalextended}
    \xymatrix{
    T \ar[r] \ar[d]^-\sigma& L \ar[d]^-\simeq \ar@{->>}[r]& H_L\ar@{^(->}[r] \ar[d]^-{\overline{\sigma}|_{H_L}} & A_T\ar[d]^-{\overline{\sigma}} \\
    T \ar[r] & M \ar@{->>}[r] & H_M \ar@{^(->}[r] & A_T
    }
\end{equation}

\subsection{K3 surfaces}
\label{ssec:K3}
Our basic reference for K3 surfaces is \cite{Huy16}.
If $X$ is a complex projective
K3 surface, $H^2(X,\Z)$ is a free abelian group of rank 22. Moreover, the cup product is a symmetric bilinear form on $H^2(X,\Z)$, turning $H^2(X,\Z)$ into an even, unimodular lattice isometric to $\Lambda_{\text{K3}} = U^{\oplus 3}\oplus E_{8}(-1)^{\oplus2}$. Here $U$ is the hyperbolic lattice given by the symmetric bilinear form $$\left(\begin{matrix}
0 & 1\\
1& 0
\end{matrix}\right),$$ and $E_8$ is the unique even, unimodular, positive-definite lattice of rank 8 (see \cite[\S VIII.1]{BPV12} for details).
The N\'eron-Severi lattice $\operatorname{NS}(X)$ is a sublattice of $H^2(X,\Z)$, defined as the image of the first Chern class $c_1: \operatorname{Pic}(X)\hookrightarrow H^2(X,\Z)$. We have $\Pic(X) \simeq \NS(X)$; it is a free abelian group of rank $\rho$ which is called the
Picard number of $X$.

The orthogonal complement $T(X) = \operatorname{NS}(X)^\perp \subseteq H^2(X,\Z)$ is called the transcendental lattice of $X$. The image of the line $H^{2,0}(X) = \CC\sigma \subset H^2(X,\CC)$ under any isometry $H^2(X,\Z)\to \Lambda_{K3}$ is called a \textit{period} of $X$. Since $\sigma^2 = 0$ and $\sigma\cdot \overline{\sigma}>0$, any period of $X$ lies in the open subset $$D \coloneqq  \left\{\ell\in \PP(\Lambda_{K3}\otimes \CC) \mid \ell^2 = 0 \text{ and } \ell\cdot \overline{\ell}>0\right\},$$ called the period domain. The following two results are among the most fundamental results about K3 surfaces.

\begin{theorem}[Surjectivity of the Period Map]\cite{Tod80}\label{thm:surjectivity-period}
Any point in the period domain is a period of a K3 surface, i.e. for any $\ell \in D$, there is a K3 surface $X$ with an isometry $H^2(X,\Z)\to \Lambda_{K3}$ such that $H^2(X,\CC)\to \Lambda_{K3}\otimes \CC$ maps $H^{2,0}(X)$ to $\ell$.
\end{theorem}

\begin{theorem}[Torelli Theorem for K3 Surfaces] 
\cite{PS71} (see \cite[Theorem 5.5.3]{Huy16})
\label{thm:torelli}
Let $X$ and $Y$ be K3 surfaces. 
Then 
$X$ and $Y$
are isomorphic
 if and only if there exists a Hodge isometry $H^2(X,\Z)\simeq H^2(Y,\Z)$. Moreover, for any Hodge isometry $\psi: H^2(X,\Z)\to H^2(Y,\Z)$ which preserves the ample cone, there is a unique isomorphism $f: X\to Y$ such that $\psi = f_*$.
\end{theorem}

The Hodge structure on the transcendental lattice determines $X$ up to derived equivalence due to what is known as the Derived Torelli Theorem.

\begin{theorem}[Derived Torelli Theorem] \cite{Muk87}, \cite{Orl03} \label{thm:derivedtorelli}
Let $X$ and $Y$ be two K3 surfaces. Then there exists an equivalence $\Db(X)\simeq\Db(Y)$ if and only if there exists a Hodge isometry $T(X)\simeq T(Y)$.
\end{theorem}

If $\Db(X) \simeq \Db(Y)$, we say that $X$ and $Y$
are derived equivalent and that $Y$ is
a Fourier-Mukai partner of $X$.
Theorem \ref{thm:derivedtorelli} implies that two derived equivalent K3 surfaces must have equal Picard numbers. 
If we denote $\Lambda = \operatorname{NS}(X)$, there is an isometry \cite[Corollary 1.6.2]{Nik80}
\begin{equation}\label{eq:A-orthog}
(A_\Lambda,\overline{q_\Lambda}) \simeq (A_{T(X)},-\overline{q_{T(X)}}).
\end{equation}
Thus derived equivalent K3 surfaces
have isomorphic discriminant lattices,
and it follows easily that
their 
N\'eron-Severi lattices 
 must be in the same genus. Instead of $(A_{T(X)},-\ol{q_{T(X)}})$, we usually write $A_{T(X)}(-1)$.

For a K3 surface $X$
we write $G_X$ for the Hodge isometries group
of $T(X)$. Then $G_X\simeq \Z/2g\Z$ for some $g\geq 1$, and we have $\phi(2g)\;|\;\operatorname{rk}T(X)$ \cite[Appendix B]{HLOY02}. 

From the Derived Torelli Theorem one can deduce:
\begin{theorem}[Counting Formula]\cite{HLOY02}
Let $X$ be a K3 surface, and write $\operatorname{FM}(X)$ for the set of isomorphism classes of Fourier-Mukai partners of $X$.
Then 
$$|\operatorname{FM}(X)| = \sum_{\Lambda}|O(\Lambda) \setminus O(A_\Lambda)/G_X|$$
where the sum runs over isomorphism
classes of lattices $\Lambda$ which are
in the same genus as the N\'eron-Severi lattice $\NS(X)$.
Furthermore, each summand computes
the number of isomorphism classes
of Fourier-Mukai partners $Y$ of $X$ with $\NS(Y) \simeq \Lambda$.
\end{theorem}

It follows from the Counting Formula
that an elliptic K3 surface $S\to \PP^1$ which admits a section has 
no non-trivial Fourier-Mukai partners \cite[Proposition 2.7(3)]{HLOY02}. 

\begin{definition} \label{def:Tgeneral}
We say that
a K3 surface $X$ is $T$-general if $G_X = \left\{\pm \operatorname{id}\right\}.$ A K3 surface that is not $T$-general is called $T$-special.
\end{definition}

When $X$ is $T$-general,
the Counting Formula shows
that the number of Fourier-Mukai partners is maximal (for a fixed $\NS(X)$)
and only depends
on $\NS(X)$.
A similar effect holds
for the invariants we study, see
Theorem \ref{thm:derivedelliptic-lagrangianelts}. Thus it is important
to have explicit criteria for $T$-generality.
If the Picard number $\rho$ of $X$ is odd, then $\phi(2g)$ must be odd,
so $|G_X| = 2$ and $X$ is $T$-general.
Furthermore, we have the following
result going back to Oguiso \cite{Ogu01}:

\begin{lemma}[{\cite[Lemma 3.9]{SZ20}}]
If $X$ is a very general K3 surface in any lattice polarized
moduli space of K3 surfaces, with Picard number
$\rho<20$, then $X$ is T-general.
\end{lemma}

See Example \ref{ex:thechosenone}
for an explicit $T$-special K3 surface.

\subsection{C\u{a}ld\u{a}raru class for a non-fine moduli space}

The Brauer group of an elliptic K3 surface with a section is one of the main technical 
tools used in this paper. We follow the discussions in \cite{Gee05} and \cite{Cal00}.
For every complex K3 surface, we have a canonical
isomorphism 
\begin{equation}\label{eq:Br-T}
\operatorname{Br}(X) \simeq \operatorname{Hom}(T(X),\Q/\Z).
\end{equation}
In particular, $\Br(X)$ is an infinite torsion group
and
for all integers $t \ge 1$ we have
\begin{equation}\label{eq:Br-tors}
\Br(X)_{t-tors} \simeq \operatorname{Hom}(T(X),\Z/t\Z) \simeq (\Z/t\Z)^{22-\rho},
\end{equation}
where $\rho$ is the Picard number of $X$.

We explain the explicit description of the
Brauer class
associated to a
moduli space of sheaves on a K3 surface \cite{Muk87}, \cite{Cal00}.
Let $X$ be a complex K3 surface, 
and consider a
Mukai vector 
\[
v = (r, D, s) \in N(X) := \Z \oplus \NS(X) \oplus \Z.
\]
We assume that 
$v$ is a primitive vector such that
$v^2 = 0$, i.e. $D^2 = 2rs$.

Let $M$ be the moduli space
of stable sheaves on $X$
of class $v$.
By Mukai's results,
if $M$ is nonempty, then
it is again a K3 surface, see e.g. \cite[Corollary 3.5]{Huy16} (we assume $v$ is primitive, so stability coincides with semistability for a generic choice of a polarization). 
Let $t$ be the divisibility of $v$, that is
\[
t = \gcd_{u \in N(X)} u \cdot v = 
\gcd\left(r, s,
\gcd_{E \in \NS(X)} E \cdot D
\right).
\]
We consider
the obstruction
$\alpha_X \in \Br(M)$
for the existence of a universal sheaf on $X \times M$;
under the isomorphism \eqref{eq:Br-T}, we will equivalently consider $\alpha_X$
as a homomorphism $T(M) \to \Q/\Z$. 
If the divisibility of $v$
equals $t$, 
then $\alpha_X$ has order $t$ 
and we have
\[
0 \to T(X) \to T(M) \overset{\alpha_X}\to \Z/t\Z \to 0.
\]
Here $\Z/t\Z$ is the subgroup of $\Q/\Z$ generated by $1/t$.
Note that the $t = 1$ case corresponds to fine moduli spaces, in which case $T(X) \simeq T(M)$.
In general we have
\begin{equation}\label{eq:extension-map}
\Z/t\Z = T(M) / T(X) \subset T(X)^* / T(X) = A_{T(X)}.
\end{equation}
We call the image $w$
of $\ol{1}$
under \eqref{eq:extension-map}
the \emph{C\u{a}ld\u{a}raru class} of $M$ (or of $v$).
By construction, the C\u{a}ld\u{a}raru class
$w$ generates the isotropic
subgroup of $A_{T(X)}$
given by Lemma \ref{lem:overlatsubgrcorresp} 
corresponding to the overlattice $T(X) \subset T(M)$.

\begin{lemma}[\cite{Cal00}]
\label{lem:caldararu-class}
Under the isomorphism \eqref{eq:A-orthog},
the C\u{a}ld\u{a}raru class $w$ 
of the
Mukai vector $(r, D, s)$
of divisibility $t$
corresponds to $-\frac1{t}D$.
\end{lemma}
\begin{proof}
By \cite[Proposition 6.4(3)]{Muk87}, the cokernel of $i:T(X)\hookrightarrow T(M)$ is generated by $\frac{1}{t}\lambda$, where $\lambda\in T(X)$ is chosen such that $D+\lambda = ta$
for some $a \in H^2(X,\Z)$.
Here $D$ and $\lambda$
correspond to each other under 
the natural isomorphism \eqref{eq:A-orthog}:
\begin{equation*}
    \begin{array}{rl@{\qquad}lr}
         A_{T(X)}(-1)&\to H^2(X,\Z)/(T(X)\oplus \NS(X)) &\to A_{\NS(X)}  \\
         \frac{1}{t}\lambda  &\mapsto \;\;\;\;\;\;\;\;\;\;\frac{1}{t}(D+\lambda) = a &\mapsto \frac{1}{t}D.
    \end{array}
\end{equation*}
Furthermore the defining equation for $\lambda$ can be equivalently written 
in the full
integral
cohomology of $X$ as
\[
v + \lambda = t\wt{a}
\]
where $\wt{a} = (r/t, a, s/t)$
(this vector is integral). We claim that $-\frac1{t}\lambda$ is the C\u{a}ld\u{a}raru class of $(r,D,s)$. To show this, we compute the value of
the Brauer class $\alpha_X$,
considered as a map 
$T(M) \to \Q/\Z$
(with image $(\frac1{t}\Z)/\Z \simeq \Z/t\Z$),
on
the element $\frac{1}{t}\lambda \in T(M)$.
Set $u \in H^*(X, \Z)$
such that $u \cdot v = 1$
(this vector exists by unimodularity).
Then we have
\[
\alpha_X(\lambda/t) = 
u \cdot \lambda/t =
u \cdot (\wt{a}-v/t) =
 u \cdot \wt{a} -u \cdot v/t
\equiv
-1/t \pmod{\Z}.
\]
Here we used
\cite[Theorem 5.3.1]{Cal00}
in the first equality and
the definition of $\wt{a}$ in the second one.
Thus we have 
$w = -\frac{1}{t}\lambda$ by definition
of the C\u{a}ld\u{a}raru class
and the corresponding element in $A_{\NS(X)}$ is $-\frac1{t}D$.
\end{proof}

\subsection{Elliptic K3 surfaces} \label{sec:ellipticsurfaces}
Recall that an elliptic surface is a surface $X$ which admits a surjective morphism $f: X\to C$ where $C$ is a smooth curve, such that the fibres of $f$ are connected and the genus of the generic fibre is $1$ \cite[\S 10]{IS96}. 
Our elliptic surfaces will
be assumed to be relatively minimal, i.e. contain no $(-1)$-curves in the fibres of $f$; this is automatic for K3 surfaces.
We say that an elliptic surface is isotrivial if all smooth fibres are isomorphic. 

For an elliptic K3 surface we have the base $C \simeq \PP^1$.
There are two natural concepts of an isomorphism
between elliptic K3 surfaces
$f: X\to \PP^1$ and $\phi: Y\to \PP^1$. 

\begin{definition}\label{def:ellip-isom}
(1) The surfaces $X$, $Y$ are isomorphic as elliptic surfaces if there exists an isomorphism
$X \simeq Y$ preserving the fibre classes, or equivalently
there is a commutative diagram
\begin{equation}\label{eq:defellipisom}
    \xymatrix{
    X\ar[r]_{{\beta}}^-\sim \ar[d]_f & Y \ar[d]^\phi\\
    \PP^1 \ar[r]_{\ol{\beta}}^-\sim & \PP^1.
    }
\end{equation}
In this case, we say that the isomorphism $X\simeq Y$ twists the base by $\ol{\beta}.$

\noindent{(2)} The surfaces $X$ and $Y$
are isomorphic over $\PP^1$ if there is an isomorphism $X\simeq Y$ twisting the base by the identity, or equivalently if there exists a commutative
diagram
\begin{equation*}
    \xymatrix{
    X\ar[rr]^-\sim \ar[dr]_-f & & Y \ar[dl]^-\phi\\
    & \PP^1. &
    }
\end{equation*}
\end{definition}

Being isomorphic over $\PP^1$ is 
more restrictive than being isomorphic
as elliptic surfaces.
For example, for every $\ol{\beta} \in \Aut(\PP^1)$,
$f: X \to \PP^1$ and $\ol{\beta} f: X \to \PP^1$
are isomorphic as elliptic surfaces, but usually
not over $\PP^1$.

Let $S\to \PP^1$ be an elliptic K3 surface with a fixed section.
We denote by
$\operatorname{Aut}_{\PP^1}(S)$ (resp. $\Aut(S, F)$)
the group of
automorphisms of $S$ over $\PP^1$ (resp. automorphisms
of $S$ preserving the fibre class).
We have $\Aut_{\PP^1}(S) \subset \Aut(S, F)$. 
We denote by $A_{\PP^1}(S)$ (resp. $A(S,F)$) the group of automorphisms of $S$ over $\PP^1$ (resp. preserving the fibre class) which also preserve the zero-section. Such automorphisms will be called \textit{group automorphisms} (see e.g. \cite{DM22}). 

\begin{remark}\label{rem:Pic0-generic}
The category
of relatively
minimal
elliptic surfaces and their isomorphisms
over $\PP^1$
is equivalent
to the category of genus one curves over $\CC(t)$
and their isomorphisms.
The functor is given by taking the generic fibre.
This functor is an equivalence
e.g. by \cite[Theorem 7.3.3]{IS96} or \cite[Theorem 3.3]{DM22}.
\end{remark}

\section{Elliptic K3 surfaces of Picard rank 2}
\label{sec:ellK3-rk2}

\subsection{N\'eron-Severi lattices} \label{sec:prelimk3lattice}

We recall some
basic facts about elliptic K3 surfaces of Picard
rank 2 following \cite{Ste02, Gee05}.
Let $f: X\rightarrow \PP^1$ be a complex projective
elliptic K3 surface.
Let $F\in \operatorname{NS}(X)$ be the class of a fibre.
Recall that the multisection index $t$
of $f$
is the minimal positive $t > 0$
such that there exists a divisor $D \in \NS(S)$
with $D \cdot F = t$.

\begin{proposition}\label{prop:neronseveriranktwo} {\cite[Remark 4.2]{Gee05}, \cite[Lemma 3.3]{SZ20}}
Let $X$ be an elliptic K3 surface
of Picard rank 2. Then there exists a polarisation $H$ on $X$ such that $H,F$ form a basis of $\operatorname{NS}(X)$ and $H\cdot F = t$. In particular, the N\'eron-Severi lattice of $X$ is given by a matrix of the form 
\begin{equation} \label{matrix:lambda}
\left(\begin{matrix}2d&t\\t&0\end{matrix}\right).
\end{equation}
\end{proposition}

We write $\Lambda_{d,t}$ for the lattice of rank 2 with matrix \eqref{matrix:lambda} with respect to some basis $H,F$.
It is easy to see that
the lattice $\Lambda_{d,t}$ has exactly 
two isotropic primitive vectors up to sign: one
is $F$, and the other is
\begin{equation}
F' = \frac1{\gcd(d,t)}(tH - dF).
\end{equation}

The following lemma describes
when the class $F'$ gives rise
to another elliptic fibration on $X$.

\begin{lemma}\cite[\S 4.7]{Gee05}\label{lem:twofibrations-numerical}
A K3 surface $X$ with $\NS(X)\simeq \Lambda_{d,t}$ has two elliptic fibrations if and only if $d\not\equiv -1\pmod t$. If $d\equiv-1\pmod t$, $X$ admits one elliptic fibration.  If $X$ is $T$-general, $t>2$ and $d\not\equiv-1\pmod t$, then the two fibrations are isomorphic (as elliptic surfaces) if and only if $d\equiv 1\pmod t$. 
\end{lemma}

We denote by $A_{d,t}$ the discriminant lattice of $\Lambda_{d,t}$
and we have
\begin{equation}\label{eq:disc-t2}
|A_{d,t}| = t^2.
\end{equation}
It is easy to compute
(see e.g. \cite[Proof of Lemma 3.2]{Ste02}) that the dual lattice $\Lambda_{d,t}^*$ 
is generated by 
\begin{equation}\label{eq:FH-star}
F^*=\frac{-2d}{t^2}F + \frac{1}{t}H, \quad H^*=\frac{1}{t}F
\end{equation}
so that the images of \eqref{eq:FH-star} generate $A_{d,t}$.
Furthermore for $a, b \in \Z$ we have
\begin{equation}\label{eq:dual-form}
q(aF^* + bH^*) = 
\frac{2a(bt-ad)}{t^2}.
\end{equation}

\begin{lemma}\label{lem:lambdabasics}
The discriminant group $A_{d,t}$
is isomorphic to $\Z/a\Z \oplus
\Z/b\Z$ with $a = \gcd(2d,t)$
and $b = t^2 / a$.
In particular, $A_{d,t}$ is cyclic if and only
if $\gcd(2d,t) = 1$.

Furthermore, if $\Lambda$ is a lattice in
the same genus as $\Lambda_{d,t}$
then $\Lambda \simeq \Lambda_{e,t}$
with $\gcd(2e,t) = \gcd(2d,t)$.
\end{lemma}
\begin{proof}
The first claim follows by putting $\Lambda_{d,t}$ into Smith normal form.

Let $\Lambda$ be a lattice in the same genus
as $\Lambda_{d,t}$.
Following the proof of \cite[Proposition 16]{HT14}, $\Lambda$ contains a primitive isotropic vector $v$. 
Hence, $\Lambda \simeq \Lambda_{e,s}$ for some $e,s \in \Z$, $s > 0$.
Comparing discriminant groups of $\Lambda_{d,t}$ and $\Lambda_{e,s}$
we obtain $t = s$ and $\gcd(2d,t) = \gcd(2e,s)$.
\end{proof}

\begin{example}\label{ex:Lambda-d0}
Let $d = 0$, then by Lemma
\ref{lem:lambdabasics},
$A_{0,t} \simeq \Z/t\Z \oplus \Z/t\Z$. Explicitly, generators
\eqref{eq:FH-star} of the dual lattice $\Lambda_{0,t}^*$
are $F^* = \frac1{t}H$
and $H^* = \frac1{t}F$
and their images in $A_{0,t}$
are the two order $t$ generators which are 
isotropic elements in $A_{0,t}$.
\end{example}

We introduce some properties of the discriminant groups which we will need to count
Fourier-Mukai partners. 

\begin{definition}\label{def:lagrangianldt}
We call an isotropic element of order $t$ in $A_{d,t}$ a Lagrangian element. 
We call a
cyclic isotropic subgroup $H\subseteq A_{d,t}$ of order $t$ a \textit{Lagrangian subgroup}.
\end{definition}

We denote by 
$\tLagr(A_{d,t})$
(resp. $\Lagr(A_{d,t})$)
the set of Lagrangian elements
(resp. Lagrangian
subgroups) of $A_{d,t}$. 
The main reason we are interested in studying Lagrangians of $A_{d,t}$ is their correspondence with Fourier-Mukai partners which we establish in Section \ref{sec:DEFM}.

\begin{proposition} \label{prop:lagrangiancount} 
Let $d,t$ be integers and let $m=\operatorname{gcd}(d,t)$. Then we have 
\begin{equation}
\label{eq:lagr-count}
|\tLagr(A_{d,t})| = \phi(t) \cdot 2^{\omega(m)}, \quad
|\Lagr(A_{d,t})| = 2^{\omega(m)}.
\end{equation}
\end{proposition}

Even though $\gcd(2d,t)$
is responsible for the structure of $A_{d,t}$,
it is
$\gcd(d,t)$ that 
appears in Proposition \ref{prop:lagrangiancount}.
For instance, 
if $d$ and $t$ are coprime and $t$ is even, the discriminant group $A_{d,t}$ is not cyclic, but $|\Lagr(A_{d,t})|=1$.

\begin{proof}
Any cyclic subgroup $H \subset A_{d,t}$
of order $t$
has $\phi(t)$ generators.
$H$ is a Lagrangian subgroup if and only if its generator is a Lagrangian element. Thus the two formulas
in \eqref{eq:lagr-count}
are equivalent, and it suffices to prove the second one.

Let $t=\prod_{p}p^{k_p}$ be the prime factorisation of $t$. 
For any prime $p$, we have an isomorphism
of $p$-adic lattices
$\Lambda_{d,t}\otimes \Z_p \simeq \Lambda_{d,p^{k_p}}\otimes \Z_p$ (the isometry is given by $H\mapsto H$ and $F\mapsto \alpha F$, where $\alpha$ is the unit  in $\Z_p$ given by $\alpha p^{k_p} = t$). 
By \cite[Proposition 1.7.1]{Nik80}, $A_{d,t}$ is isometric to the orthogonal direct sum of $A_{d, p^{k_p}}$ over all primes $p$.
Therefore we have $$|\Lagr(A_{d,t})| = \prod_{p}|\Lagr(A_{d,p^{k_p}})|.$$
Therefore we need to prove that $|\Lagr(A_{d,p^k})| = 1$ if $d$ is coprime to $p$ and $|\Lagr(A_{d,p^k})|=2$ otherwise. The result follows from Lemma \ref{lem:twolagrangians-equal} and Lemma \ref{lem:lagrangians-prime} below.
\end{proof}

\begin{lemma}\label{lem:twolagrangians-equal}
The elements
\begin{equation}\label{eq:v-vprime}
v = \frac{1}{t}F, \quad v' = \frac{1}{t}F'
\end{equation}
are primitive isotropic vectors in $\Lambda_{d,t}^*$
and their images $\ol{v}$ and $\ol{v'}$
in $A_{d,t}$ generate Lagrangian subgroups in $A_{d,t}$.
We have $\langle \ol{v} \rangle = \langle \ol{v'} \rangle$ if and only if $m\coloneqq\gcd(d,t)=1$, 
in which case 
\begin{equation}\label{eq:v'-dv}
\ol{v'} = -d\cdot\ol{v}    
\end{equation}

\end{lemma}
\begin{proof}
The first part is a simple computation.
The corresponding Lagrangian subgroups are equal if and only if $v' = \frac{1}{tm}(tH-dF) = \frac{1}{m}H - \frac{d}{tm}F$ is a multiple of $v = \frac{1}{t}F$ modulo $\Lambda_{d,t}$. This is only the case when $m=1$.  
\end{proof}

\begin{lemma}\label{lem:lagrangians-prime}
Let $t = p^k$ with $p$ a prime number and $k\geq 1$. 
Then the subgroups $\langle \ol{v} \rangle, \langle \ol{v'} \rangle$  
are the only Lagrangian subgroups of $A_{d,t}$.
\end{lemma}

\begin{proof}
Write $d = \ell\cdot p^n$ for some $\ell\in \Z$ coprime to $p$ and some $n\geq 0$. Note that whenever $n\geq k$, we have $d\equiv 0\pmod {p^k}$, so that $\Lambda_{d,p^k}\simeq \Lambda_{0,p^k}$ and we can assume that $d = 0$. In this case we have $\ol{v'} = \ol{F^*}$ and it is easy to see that $\langle\ol{H^*}\rangle$ and $\langle\ol{F^*}\rangle$ are the only Lagrangian subgroups of $A_{0,p^k}$
(see Example \ref{ex:Lambda-d0}). 
Therefore we may assume $0\leq n < k$. 

In terms of generators
\eqref{eq:FH-star}
the quadratic form is given by 
\begin{equation}\label{eq:quadraticformoriginal}
q(aF^* + bH^*) = \frac{2a}{p^{2k-n}}\left(bp^{k-n}-a\ell\right). 
\end{equation}

To find all Lagrangian
subgroups, we start by describing
the subgroup of elements in $A_{d,t}$
having order dividing $t = p^k$.
We consider the vectors \eqref{eq:v-vprime} which in our case are given by
\[
\ol{v} = \frac{F}{p^k}, \quad
\ol{v'} = \frac{H}{p^n} - \frac{\ell F}{p^k}.
\]
Furthermore, the orders of $\ol{v}$
and $\ol{v'}$ are equal to $p^k$,
and these elements satisfy a relation
\begin{equation}\label{eq:u-v-rel}
p^n (\ell \ol{v} + \ol{v'}) = 0.    
\end{equation}
There are two cases to consider now.
If $p > 2$, then
\[
(A_{d,t})_{p^k-tors} = \left\langle \frac{F}{p^k},
\frac{H}{p^n} \right\rangle = \langle \ol{v}, \ol{v'} \rangle.
\]

The vectors $\ol{v}$ and
$\ol{v'}$ are isotropic
and
the discriminant form in terms
of 
these elements
equals
\[
\ol{q}(a\ol{v} + b\ol{v'}) = \frac{2ab}{p^n}.
\]
Hence an element $a \ol{v} + b \ol{v'}$ is isotropic if and only 
$p^n$ divides $ab$.
On the other hand, if
$a \ol{v} + b \ol{v'}$ has order precisely
$p^k$, then at least one of $a$
or $b$ is coprime to $p$.
Hence isotropic elements of $A_{d,t}$
of order $p^k$ are given by
\[
a \ol{v} + b p^{n + j} \ol{v'},
\quad
a p^{n + j} \ol{v} + b \ol{v'},
\]
with both $a$ and $b$ coprime to $p$
and $j \ge 0$. Using
\eqref{eq:u-v-rel}
we can rewrite these types of elements
as
\[
a' \ol{v},
\quad
b' \ol{v'},
\]
with $a'$ and $b'$ coprime to $p$. This finishes the proof in the $p > 2$ case.

If $p = 2$, then 
\[
\frac{1}{2^{n+1}}H\cdot F = \frac{2^k}{2^{n+1}} \quad\text{and}\quad \frac{1}{2^{n+1}}H^2 = 2\ell\cdot \frac{2^{n}}{2^{n+1}}
\]
are both integers. This means that $\frac{1}{2^{n+1}}H$ is an element of $A_{d,2^{k}}$ by \eqref{eq:characterisingduallattice}, and we have
\[
(A_{d,t})_{2^k-tors} = \left\langle \frac{F}{2^k},
\frac{H}{2^{n+1}} \right\rangle 
\supsetneq 
\langle \ol{v}, \ol{v'} \rangle =
\left\langle\frac{F}{2^k},
\frac{H}{2^{n}} \right\rangle.
\]
However, a simple computation shows that
all isotropic vectors are actually contained
in $\langle \ol{v}, \ol{v'} \rangle$
and the proof works in the same way as in the $p > 2$ case.
\end{proof}

Lemma \ref{lem:lagrangians-prime}
allows us to define a canonical involution
on the set of Lagrangian subgroups of $A_{d,t}$
as follows.
For $H \subset A_{d,t}$ a Lagrangian, we 
take its primary decomposition with respect
to \eqref{eq:A-primary}
\[
H = \bigoplus_p H_p, \quad H_p \subset A_{d,t}^{(p)}
\]
with each $H_p$ a Lagrangian in $A_{d,t}^{(p)}$.
We set $\iota_p(H_p)$ to denote the other
Lagrangian subgroup as determined by Lemma
\ref{lem:lagrangians-prime}; in the case $p$
does not divide $d$,
$\iota_p(H_p) = H_p$.
We set
\begin{equation}\label{eq:involution}
\iota(H) := \bigoplus_p \iota_p(H_p) \subset A_{d,t}.
\end{equation}
The geometric
significance of this involution is explained
in Theorem \ref{thm:derivedelliptic-lagrangianelts}. For now we note that
\begin{equation}\label{eq:iota-v-vprime}
\iota(\langle \ol{v} \rangle) = 
\langle \ol{v'} \rangle
\end{equation}
for $\ol{v}$, $\ol{v'}$ defined in Lemma \ref{lem:twolagrangians-equal}.

\subsection{Automorphisms and Hodge isometries}

Recall the Hodge isometry group $G_X$
defined in Section \ref{ssec:K3}.

\begin{lemma}\label{lem:orders-G}
If $X$ is a K3 surface of Picard rank 2, then $G_X$ is a cyclic group
of one of the following orders:
\[
2,
4,
6,
8,
10,
12,
22,
44,
50,
66.
\]
\end{lemma}
\begin{proof}
The fact that $G_X$ is a finite cyclic group of even order $2g$ such that $\phi(2g)|\operatorname{rk}T(X)$ is proved in \cite[Appendix B]{HLOY02}.
We solve the equation $\phi(2g) \; | \; 20$.
Possible primes that can appear in the prime factorization of $2g$ are $2$, $3$, $5$, $11$. Maximal powers of these
primes such that $\phi(p^k) \; | \; 20$
are $2^3$, $3$, $5^2$, $11$
and the result follows by combining these or smaller
prime powers.
\end{proof}

\begin{proposition} \label{prop:autsequence}
Let $X$ be an elliptic
K3 surface of Picard rank 2. 
Then we have a canonical isomorphism
\begin{equation}\label{eq:eq-Aut-main}
\Aut(X) \simeq \Ker\left(G_X \to O(A_{T(X)})/O^+(\NS(X))\right),
\end{equation}
where $O^+(\NS(X))$ is the group of isometries of $\NS(X)$ that preserve the ample cone. In particular, $\operatorname{Aut}(X)$ is a finite cyclic group
and $|\Aut(X)| \le 66$.
Moreover, for any elliptic fibration $ X\to \PP^1$, the isomorphism above induces an isomorphism 
\begin{equation}\label{eq:eq-Aut-F}
\Aut(X,F) \simeq \Ker\left(G_X \to O(A_{T(X)})\right),
\end{equation}
where $\operatorname{Aut}(X,F)$ is the group of 
automorphisms which fix the fibre class $F$ of the elliptic fibration.
\end{proposition}
\begin{proof}
By the Torelli Theorem \ref{thm:torelli}, 
there is a bijection 
between automorphisms of $X$
and Hodge isometries 
of $H^2(X,\Z)$ 
which preserve the ample cone.
Using 
\cite[Corollary 1.5.2]{Nik80},
we can write
\begin{equation}\label{eq:Aut-Torelli}
\Aut(X) \simeq \left\{(\sigma,\tau)\in G_X\times O^+(\NS(X))\;\mid\;\ol{\sigma}=\ol{\tau}\in O(A_{T(X)})\right\}.
\end{equation}
This isomorphism induces a surjective map $(\sigma, \tau) \mapsto \sigma$
\begin{equation}\label{eq:Autnew}
\Aut(X)\to\Ker\left(G_X\to O(A_{T(X)})/O^+(\NS(X))\right).
\end{equation}
The kernel of this map consists of the pairs $(\operatorname{id}_{T(X)},\tau)\in G_X\times O^+(\NS(X))$ such that $\ol{\tau}=\operatorname{id}_{A_{T(X)}}.$

We claim that
the homomorphism $O^+(\NS(X)) \to O(A_{T(X)})$ is injective. 
Since $\NS(X)$ contains
four isotropic vectors $\pm F, \pm F'$,
and $-1 \in O(\NS(X))$ never preserves
the ample cone, 
we note that $O^+(\NS(X))$ must be either trivial,
or isomorphic to 
$\Z/2\Z$ with non-trivial
element swapping $F$ with $F'$. 
The latter case is only possible
when $F'$ represents
a class
of an elliptic fibration on $X$,
which by Lemma \ref{lem:twofibrations-numerical}
corresponds to the case
$d \not\equiv -1 \pmod{t}$.
Then $\frac1{t}F$ and $\frac1{t}F'$
represent distinct
classes in $A_{T(X)}$ (see \eqref{eq:v'-dv})
and the element of $O^+(\NS(X))$
swapping $F$ and $F'$ has a non-trivial
image in $O(A_{T(X)})$.
Thus, since $O^+(\NS(X)) \to O(A_{T(X)})$
is injective, the map \eqref{eq:Autnew} is a bijection.
The claim about the isomorphism
type of $|\Aut(X)|$
follows from Lemma \ref{lem:orders-G}.
For the last statement, note that the only element of $O^+(\NS(X))$ 
which fixes $F$ is the identity. Therefore, 
\eqref{eq:eq-Aut-F} also follows from 
\eqref{eq:Aut-Torelli}.
\end{proof}

\begin{example}
Let $X$ be an elliptic K3 surface
with $\NS(X)\simeq \Lambda_{d,t}$
and assume that
$\gcd(2d,t)=1$.
In this case $A_{d,t}$ is cyclic of order $t^2$
by Lemma \ref{lem:lambdabasics}. An isometry $\sigma\in O(A_{d,t})$ is given by multiplication by a unit $\alpha\in \Z/t^2\Z$ with $\alpha^2 \equiv 1 \pmod t$, so
that the group $O(A_{d,t})$ is $2$-torsion. 
Thus by Proposition \ref{prop:autsequence}, $\Aut(X) \subset G_X$
is a cyclic subgroup of index one or two.
\end{example}

\begin{lemma} \label{lem:autgroupwithsection}
Let $S$ and $S'$ be K3 surfaces of Picard rank 2 which admit elliptic fibrations with a section. Then every Hodge isometry
between $T(S)$ and $T(S')$ lifts to a unique
isomorphism between $S$ and $S'$.
In particular, we have $\Aut(S) \simeq G_S$.
Finally, $S$ admits a unique
elliptic fibration with a unique section, hence
every automorphism of $S$ is a group automorphism.
\end{lemma}
\begin{proof}
By Proposition \ref{prop:neronseveriranktwo}
we have $\NS(S)\simeq \Lambda_{d,1}$,
which is isomorphic to the 
hyperbolic lattice $U$,
in particular
$\NS(S)$ is unimodular and 
$A_{\NS(S)} =0$.
If there is a Hodge isometry between
$T(S)$ and $T(S')$, 
extending it to a Hodge isometry
between $H^2(S, \Z)$ and $H^2(S', \Z)$
preserving the ample cones,
we obtain $S \simeq S'$,
by the Torelli Theorem,
as in the proof of 
Proposition \ref{prop:autsequence}.
Thus we may assume that $S = S'$
in which case the result follows 
Proposition \ref{prop:autsequence}.

By Lemma \ref{lem:twofibrations-numerical},
$S$ admits a unique elliptic fibration.
Since $\NS(S) = U$,
there is a unique $(-2)$-curve which intersects the fibres of the elliptic fibration
with multiplicity $1$, i.e. a unique section.
\end{proof}

\begin{example} \label{ex:thechosenone}
Let $S\to \PP^1$ be the elliptic K3 surface with a section given by the Weierstrass equation $y^2 = x^3 + t^{12} - t$. This surface is isotrivial with $j$-invariant $0$. It was studied in \cite{Keu16} and \cite{Kon92}. We have $\rk \NS(S) = 2$, and $S$ is $T$-special. In fact, the group $G_S$ is cyclic of order $66$, and $S$ is unique with this property. Furthermore, 
$\Aut(S) \simeq \Z/66\Z$ by Lemma \ref{lem:autgroupwithsection}. 
The action of the
subgroup $\Z/6\Z \subset \Aut(S)$ 
commutes with projection to $\PP^1$ and rescales
$x$ and $y$ coordinates, and
the subgroup $\Z/11\Z \subset \Aut(S)$ preserves the fibre class $F \in \NS(S)$ and induces an order $11$
automorphism $t \mapsto \zeta_{11} t$
on $\PP^1$.
\end{example}

\begin{corollary} \label{cor:autfixingfibre}
Let $X$ be a T-general elliptic K3 surface of Picard rank 2 and multisection index $t>2$, then $\operatorname{Aut}(X,F) = \left\{\operatorname{id}\right\}.$
\end{corollary}
\begin{proof}
By Proposition \ref{prop:autsequence}, there is an isomorphism $\Aut(X,F)\simeq \operatorname{Ker}(G_X\to O(A_{T(X)})).$ 
We have $G_X = \{\pm 1\}$ by assumption. 
Since
$t > 2$, and 
$A_{T(X)}$ has order $t^2$, we see that
$-1$ acts non-trivially on $A_{T(X)}$.
Thus $\Ker(G_X\to O(A_{T(X)}))$ is trivial.
\end{proof}

\section{Jacobians}
\label{sec:Jacobians}

\subsection{Ogg-Shafarevich Theory} 
\label{sec:brauerintro}

Given an elliptic K3 surface $f:X \to \PP^1$
and $k \in \Z$
we can define an elliptic K3 surface $\Jac^k(f): \Jac^k(X) \to \PP^1$, called the $k$-th Jacobian of $X$,
as the moduli space 
of sheaves supported at the fibres of $f$
and having degree $k$ 
\cite[Chapter 11]{Huy16}.
In particular, we have $S := \Jac^0(X)$ which
is an elliptic K3 surface with a distinguished section.

In what follows, we sometimes write $C$, $C'$ for bases of elliptic fibrations when they are not canonically isomorphic.

\begin{lemma} \label{lem:functorialjaczero}
Let $X\to C$ and $X'\to C'$ be elliptic K3 surfaces with zeroth Jacobians $S\to C$ and $S'\to C'$, respectively. Then an isomorphism of elliptic surfaces $\gamma:X\simeq X'$ which twists the base by $\ol{\beta}:C\to C'$ (see Definition \ref{def:ellip-isom}), induces a group isomorphism $\Jac^0(\gamma):S\simeq S'$ twisting the base by $\ol{\beta}$.
\end{lemma}

\begin{proof}
When $\ol{\beta}$ is the identity,
this is a standard result
which follows immediately
from Remark \ref{rem:Pic0-generic}.
For the general case,
see \cite[\S3, (3.3)]{DM22}.
\end{proof}

The Ogg--Shafarevich theory
relates elements 
in the Brauer group $\Br(S)$
of an elliptic K3 surface
$S$ with a section, to $S$-torsors.
For our purposes, the following definition
of a torsor is convenient. See 
\cite{Fri98},
\cite[Proposition 5.6]{Huy16}
for the equivalence with the standard definition.

\begin{definition}Let $f:S\to \PP^1$ be an elliptic K3 surface with a section. An \textit{$f$-torsor} 
is a pair $(g: X\to\PP^1,\theta)$ where $g: X\to \PP^1$ is an elliptic K3 surface and $\theta: \Jac^0(X)\to S$ is an isomorphism over $\PP^1$ preserving the zero-sections, i.e. a group isomorphism over $\PP^1$. 
\end{definition}

An isomorphism of $f$-torsors $(g: X\to \PP^1, \theta)$ and $(h:Y\to \PP^1, \eta)$ is an isomorphism $\gamma: X\to Y$ over $\PP^1$ such that 
$$
\xymatrix{
\Jac^0(X) \ar[rr]^-{\Jac^0(\gamma)} \ar[dr]_-\theta && \Jac^0(Y) \ar[dl]^-\eta \\ & S &
}
$$ commutes.

\begin{example}\label{ex:X-J-torsor}
If $X$ is an elliptic K3 surface,
then $X$ has a natural structure $(X, \mathrm{id}_{\Jac^0(X)})$
of a 
torsor over $\Jac^0(X)$.
Since $\Jac^0(\Jac^k(X)) = \Jac^0(X)$ (this can be checked e.g.
using Remark \ref{rem:Pic0-generic}),
all Jacobians $\Jac^k(X)$ also have a natural
$\Jac^0(X)$-torsor structure.
\end{example}

The set of isomorphism classes of $f$-torsors
is in bijection with the Tate--Shafarevich group of $f: S\to \PP^1$
\cite[11.5.5(ii)]{Huy16}, and we denote it $\Sha(f: S\to \PP^1)$ or just $\Sha(S)$ if it can not lead to confusion.
If $S\to \PP^1$ is an elliptic K3 surface with a section, then there is an isomorphism
\begin{equation}\label{eq:Br-Sha}
\operatorname{Br}(S)\simeq \Sha(S),
\end{equation}
see \cite{Cal00}, \cite[Corollary 11.5.5]{Huy16}.
We recall
the construction of the Tate--Shafarevich group 
and of \eqref{eq:Br-Sha}
in the proof of Lemma \ref{lem:functorialitytorsors}.
For an $S$-torsor $(X, \theta)$ we write $\alpha_X \in \Br(S)$
for the class corresponding to $[X] \in \Sha(S)$
under \eqref{eq:Br-Sha}. 
It would be more precise to include $\theta$ in the notation, but we do not do that, assuming that
the torsor structure on $X$ is fixed.
We also write $\alpha_X: T(X) \to \Q/\Z$ for the corresponding
element with the respect to \eqref{eq:Br-T}.

\begin{lemma}\label{lem:Br-Sha-WC}
Let $(X, \theta)$ be an $S$-torsor. Let $t$ be the order of $\alpha_X \in \Br(S)$.

(i) $X$ has a section if and only if $\alpha_X = 0$,
in which case $X$ is isomorphic to $S$
as an $S$-torsor.

(ii) For all $k \in \Z$ we have
$\alpha_{\Jac^k(X)} = k\cdot \alpha_X$.

(iii) The multisection index of $X$ equals $t$.

(iv) We have a Hodge isometry
$T(X) \simeq \Ker(\alpha_X: T(S) \to \Z/t\Z)$.
\end{lemma}
\begin{proof}

(i) 
It follows by construction that
all $S$-torsor structures on $S$
are isomorphic, and 
correspond to
$0 \in \Br(S)$
under 
\eqref{eq:Br-Sha}.
Thus, if $\alpha_X = 0$, then $X$
is isomorphic as $S$-torsor to $S$,
in particular $X$ and $S$ 
are isomorphic as elliptic
surfaces, hence $X$ has a section.
Conversely, if $X$ has a section, then
we have $S \simeq \Jac^0(X) \simeq X$
hence $X$ is isomorphic as a torsor to some
torsor structure on $S$, so that $\alpha_X = 0$
by the argument above.

Part (ii) is \cite[Theorem 4.5.2]{Cal00}
and
part (iv) is \cite[Theorem 5.4.3]{Cal00}. 

(iii) For a K3 surface $X$ with a chosen elliptic fibration
let us write $\ind(X)$
for the multisection index of the fibration. 
Since $\Jac^{\ind(X)}(X)$ admits
a section, we have $\Jac^{\ind(X)}(X)\simeq S$ as torsors by (i).
It follows using (ii)
that $0 = \alpha_{\Jac^{\ind(X)}(X)} = \ind(X)\alpha_X$
hence
$\ord(\alpha_X)$ divides $\ind(X)$.
To prove their equality we use
\cite[Ch. 4, (4.5), (4.6)]{Huy06} 
to deduce that for all $k \in \Z$
\[
\ind(\Jac^k(X)) = \frac{\ind(X)}{\gcd(\ind(X), k)}.
\]
In particular, 
\[
1 = \ind(\Jac^{\ord(\alpha_X)}(X)) = \frac{\ind(X)}{\gcd(\ind(X), \ord(\alpha_X))} = \frac{\ind(X)}{\ord(\alpha_X)}
\]
so that $\ind(X) = \ord(\alpha_X)$, which proves part (ii).
\end{proof}

Let $S\to \PP^1$ be an elliptic K3 surface with a section.
Recall that we denote by
$A_{\PP^1}(S)$ (resp. $A(\PP^1, F)$)
the group 
of group automorphisms of $S$ over $\PP^1$ (resp. group automorphisms
of $S$ preserving the fibre class $F \in \NS(S)$).
We have $A_{\PP^1}(S) \subset A(S, F)$, and we are interested in the orbits of these two groups acting on the Brauer group $\operatorname{Br}(S)$.
We do this more generally, by
explaining functoriality of $\Sha(S)$ and $\Br(S)$
 with respect to $S$.

Let $f:S\to C$ and $f':S'\to C'$ be elliptic K3 surfaces with fixed sections. 
Assume that there exists a group isomorphism $\beta: S \simeq S'$ twisting the base by $\overline{\beta}: C \simeq C'$. We define a map $\beta_*: \Sha(f: S\to C)\to \Sha(f':S'\to C')$ as follows: $$\beta_*(g: X\to C, \theta) = (\overline{\beta}\circ g: X\to C', \beta\circ\theta).$$ 
Note that 
the element on the right-hand side
belongs to $\Sha(f')$ by Lemma
\ref{lem:functorialjaczero}. 

Furthermore, in the same setting,
we define $\beta_*:\operatorname{Hom}(T(S),\Q/\Z) \to \operatorname{Hom}(T(S'),\Q/\Z)$ by $\beta_*(\alpha) = \alpha\circ \beta^*$, where $\beta^*: T(S')\to T(S)$ is the Hodge isometry induced by $\beta$.
It is important for applications that these two pushforwards are compatible with \eqref{eq:Br-Sha}:

\begin{lemma} \label{lem:functorialitytorsors}
Let $f: S\to C$ and $f': S'\to C'$ be elliptic K3 surfaces with fixed sections, and let $\beta:S\simeq S'$ be a group isomorphism twisting the base by $\ol{\beta}$. Then there is a commutative square of isomorphisms 
\begin{equation}\label{eq:functorialtorsors}
\xymatrix{
\Sha(f:S\to C) \ar[r]^-{\beta_*} \ar[d] & \Sha(f':S'\to C')  \ar[d]\\
\operatorname{Hom}(T(S),\Q/\Z) \ar[r]^-{\beta_*} & \operatorname{Hom}(T(S'),\Q/\Z).
}
\end{equation}
\end{lemma}
\begin{proof}
The vertical arrows in  \eqref{eq:functorialtorsors} are the compositions of the vertical maps in the following diagram, with cohomology groups in  and analytic topology respectively: 
\begin{equation}\label{eq:ladder}
    \xymatrix{
    \Sha(f: S\to C) \ar[r]^-{\beta_*} \ar[d] & \Sha(f': S'\to C') \ar[d] \\
    H^1(C,\sheaf{X}_0) \ar[r]^-{(1)} \ar[d] & H^1(C',\sheaf{X}'_0) \ar[d]\\ 
    H^2(S,\mathbb{G}_m) \ar[r]^-{(2)} \ar[d]_-{} & H^2(S',\mathbb{G}_m) \ar[d]^-{} \\ 
    H^2_{an}(S,\mathcal{O}^*_S)_{tors} \ar[r]^-{(3)} \ar[d]_-{(4)} & H^2_{an}(S',\mathcal{O}^*_{S'})_{tors} \ar[d]^-{(4)} \\ 
    \operatorname{Hom}(T(S),\Q/\Z) \ar[r]^-{\beta_*} & \operatorname{Hom}(T(S'),\Q/\Z),
    }
\end{equation}
c.f. \cite[Corollary 11.5.6]{Huy16}. Here $\sheaf{X}_0$ and $\sheaf{X}_0'$ are the sheaves of \'etale local sections of $f$ and $f'$, respectively. The horizontal arrows (1), (2), (3) are induced by $\overline{\beta}_*\sheaf{X}_0\simeq \sheaf{X}_0'$ and $\beta_*\mathbb{G}_m \simeq \mathbb{G}_m.$ Arrows (4) are induced by the exponential sequence. One can check commutativity for each square in  \eqref{eq:ladder}, and this gives the desired result.
\end{proof}

\begin{proposition}\label{prop:functorial--Sha}
Let $f: S \to C$, $f': S' \to C'$ be elliptic K3 surfaces with sections.
Let $(g: X \to C,\theta)$, $(g': X' \to C',\theta')$ be torsors for $f$ and $f'$ respectively. Then there is a group isomorphism $\beta:S\simeq S'$, twisting the base by $\ol{\beta}: C\simeq C'$ and such that $\beta_*(g,\theta) \simeq (g',\theta')$ if and only if there is an elliptic surface isomorphism $X\simeq X'$ twisting the base by $\ol{\beta}.$
\end{proposition}
\begin{proof}
Suppose there is a group isomorphism $\beta: S\simeq S'$ twisting the base by $\ol{\beta}$ and such that $\beta_*(g,\theta) = (g',\theta').$ Then it follows from the definition of $\beta_*$ that there is an elliptic surface isomorphism $X\simeq X'$ twisting the base by $\ol{\beta}.$
Conversely, suppose there is an elliptic surface isomorphism $\gamma: X\simeq X'$ twisting the base by $\ol{\beta}$. 
Consider the isomorphism $\beta\coloneqq \theta'\circ \Jac^0(\gamma)\circ \theta^{-1}: S\to S'$.
We can compute $\beta_*(g,\theta)$,
decomposing $\beta_*$
as a composition of isomorphisms
\[
\Sha(S) \overset{\theta^{-1}_*}{\to} \Sha(\Jac^0(X))
\xrightarrow{\Jac^0(\gamma)_*}
\Sha(\Jac^0(X')) \overset{\theta'_*}\to \Sha(S')
\]
to see that $\beta_*(g,\theta) = (g',\theta').$
\end{proof}

\begin{remark}
The proof of Proposition \ref{prop:functorial--Sha} in fact shows that given $(g,\theta)$, $(g', \theta')$
as in the statement, the set of isomorphisms between elliptic fibrations
$g$
and $g'$ twisting the base by $\ol{\beta}$
(and ignoring the choice 
of $\theta$, $\theta'$)
is in natural bijection with the set
of group isomorphisms $\beta$ between $S$
and $S'$ twisting the base by $\ol{\beta}$
together with a chosen isomorphism
$\gamma$ between $\beta_*(g, \theta)$
and $(g', \theta')$.
\end{remark}

It will be more convenient for us to work
with the Brauer group instead of the Tate--Shafarevich group:

\begin{proposition}\label{prop:functorialitytorsors}
Using the same notation as in Proposition \ref{prop:functorial--Sha}, there is a group isomorphism $\beta:S\simeq S'$, twisting the base by $\ol{\beta}: C\simeq C'$ and such that $\beta_*\alpha_X = \alpha_{X'}$ if and only if there is an elliptic surface isomorphism $X\simeq X'$ twisting the base by $\ol{\beta}.$
\end{proposition}
\begin{proof}
This follows immediately from Proposition \ref{prop:functorial--Sha} and Lemma \ref{lem:functorialitytorsors}.
\end{proof}

\begin{corollary}
\label{cor:Jk-funct}
Let $g: X \to C$, $g': X' \to C'$ be elliptic K3 surfaces which  are isomorphic via an isomorphism which twists the base by $\overline{\beta}: C \to C'$.
Then for all $k \in \Z$, there exists an elliptic surface isomorphism $\Jac^k(X)\simeq\Jac^k(X')$ twisting the base by $\ol{\beta}.$
\end{corollary}
\begin{proof}
Let $S\to C$ and $S'\to C'$ be the zeroth Jacobians of $X\to C$ and $X'\to C'$, respectively. By Proposition \ref{prop:functorialitytorsors}, there is a group isomorphism $\beta:S\to S'$ such that $\beta_*\alpha_X = \alpha_{X'}$. This means that $\beta_*(k\cdot\alpha_X) =k\cdot \beta_*\alpha_X =  k\cdot \alpha_{X'}$ for all $k\in \Z$. Since the Brauer classes of $\Jac^k(X)\to C$ and $\Jac^k(X')\to C'$ are $k\cdot\alpha_X$ and $k\cdot\alpha_{X'}$, the result follows from Proposition \ref{prop:functorialitytorsors}.
\end{proof}

\begin{corollary}\label{cor:orbitsareisomclasses}
Let $S\to C$ be an elliptic K3
surface with a section. 
The set of $A(S,F)$-orbits 
(resp. $A_C(S)$-orbits)
of $\Br(S)$
parametrizes 
$S$-torsors
up to isomorphism as elliptic surfaces
(resp. up to isomorphism over $C$).
\end{corollary}
\begin{proof}
We put $S = S'$ in Proposition
\ref{prop:functorialitytorsors},
consider $S$-torsors $(X, \theta)$ and $(X', \theta')$
and write $\alpha_X, \alpha_{X'} \in \Br(S)$
for the corresponding Brauer classes.
By Proposition
\ref{prop:functorialitytorsors}
there is an isomorphism
between 
elliptic surfaces $X$, $X'$
twisting the base (resp. over the base)
if and only if there exists 
$\beta \in A(S, F)$ (resp. 
$\beta \in A_C(S)$) such that $\beta_*(\alpha_X) = \alpha_{X'}$.
Thus the resulting sets of orbits
are as stated in the Corollary.
\end{proof}

\begin{example}\label{ex:dual-torsor}
The automorphism $\beta = -1 \in A_C(S)$
acts on $\Br(S)$ as multiplication by $-1$.
This way we always have (at least) two
torsor structures on every elliptic K3 surface $X$.
If $X$ has no sections, 
these two 
torsor structures are isomorphic if and only if $\alpha_X \in \Br(X)$ has order two, which by Lemma \ref{lem:Br-Sha-WC} is equivalent to $X$ having multisection index two.  
\end{example}

We write $\EllK3$ for the set of isomorphism classes of elliptic K3 surfaces (isomorphisms are allowed
to twist the base). 
We can express Ogg--Shafarevich theory
as
a natural bijection between $\EllK3$
and the set of isomorphism
classes
of twisted Jacobian K3 surfaces.

\begin{definition}\label{def:twisted-jacobian-K3}
A \textit{twisted Jacobian K3 surface}
is a triple
$(S,f,\alpha)$ where $S$ is a K3 surface with elliptic fibration $f$ together with a fixed section, and $\alpha$ is a Brauer class on $S$.
\end{definition}

An isomorphism of two twisted Jacobian K3 surfaces $(S,f: S \to C,\alpha)$ and $(S',f': S' \to C',\alpha')$ is a group isomorphism $\beta:S\simeq S'$ such that $\beta_*\alpha=\alpha'.$
We write $\BrK3$ for the set of isomorphism classes of
twisted Jacobian K3 surfaces.
The above results show the following.

\begin{theorem}\label{thm:ell--br}
The map $\EllK3\to\BrK3$ given by $(X,g)\mapsto (\Jac^0(X),\Jac^0(g),\alpha_X)$ is a bijection.
\end{theorem}
\begin{proof}
From Proposition \ref{prop:functorialitytorsors}, it follows that the map $\EllK3\to \BrK3$ is well-defined and injective. For the surjectivity, let $(S,f,\alpha)\in \BrK3$. Using the isomorphism \eqref{eq:Br-Sha}, we obtain an $S$-torsor $(g:X\to\PP^1,\theta:\Jac^0(X)\simeq S)\in \Sha(f)$ corresponding to $\alpha$. In particular, the map $\EllK3\to \BrK3$ 
assigns $(X,g)\mapsto (\Jac^0(X),\Jac^0(g),
\theta^{-1}_*\alpha)\simeq (S,f,\alpha).$ 
\end{proof}

\subsection{Isomorphisms of Jacobians}\label{sec:countjac}

We work with an elliptic K3 surface $X$; recall
from Example \ref{ex:X-J-torsor}
that $X$ and all its
Jacobians
$\Jac^k(X)$ have a natural
structure of a torsor
over $\Jac^0(X)$.

\begin{lemma}\label{lem:JkJl} \cite[Theorem 4.5.2]{Cal00}
Let $X$ be an elliptic K3 surface, and let $k,\ell\in \Z$. Then we have $\Jac^k(\Jac^\ell(X)) \simeq \Jac^{k\ell}(X)$ as torsors over $\Jac^0(X)$.
\end{lemma}
\begin{proof}
By Lemma \ref{lem:Br-Sha-WC}, 
in the Tate--Shafarevich group of $\Jac^0(X)$, we have $[\Jac^k(\Jac^\ell(X))] = k\cdot [\Jac^\ell(X)] = k\ell\cdot [X] = [\Jac^{k\ell}(X)].$ In particular, we have $\Jac^k(\Jac^\ell(X)) \simeq \Jac^{k\ell}(X)$ as torsors over $\Jac^0(X)$.
\end{proof}

Let $t$ be the multisection index of $X$.
We are especially interested in those Jacobians
for which $\gcd(k, t) = 1$.
We call these \emph{coprime Jacobians} of $X$.
By Theorem \ref{thm:coprimejacfm} below,
every coprime Jacobian is a Fourier-Mukai
partner of $X$.
For all $k \in \Z,$ we have well-known 
isomorphisms over $\PP^1$:
\begin{equation}\label{eq:standardjacobianisomorphisms}
\Jac^{k + t}(X) \simeq \Jac^k(X), \quad 
\Jac^{-k}(X) \simeq \Jac^k(X).
\end{equation}
Here the first isomorphism follows by adding the multisection on the generic fibre, and then spreading out
as in Remark \ref{rem:Pic0-generic},
and the second isomorphism can be obtained, by the same token, from the dualization of line bundles, or alternatively deduced from Proposition \ref{prop:functorialitytorsors}
with $\beta$ acting by $-1$ on the fibres
(see Example \ref{ex:dual-torsor}).

We see that there are 
at most $\phi(t)/2$ 
isomorphism classes of coprime Jacobians of $X$.
The goal of the next result
is to 
be able to compute this number precisely, see \eqref{eq:jspec-formula}
for what this count will look like.

\begin{proposition}\label{prop:jacobians-isom}
Let $X \to \PP^1$ be an elliptic K3 surface of multisection index $t > 2$. Then $\Jac^k(X) \simeq \Jac^\ell(X)$ as $\Jac^0(X)$-torsors 
if and only if $k \equiv \ell \pmod{t}$.
Furthermore there exist subgroups $B_X \subset \wt{B}_X \subset (\Z/t\Z)^*$,
such that for $k, \ell \in (\Z/t\Z)^*$ we have
\[
\Jac^k(X) \simeq \Jac^\ell(X) \text{ over $\PP^1$} \iff k \ell^{-1} \in B_X,
\]
and
\[
\Jac^k(X) \simeq \Jac^\ell(X) \text{ as elliptic surfaces} \iff k \ell^{-1} \in \widetilde{B}_X.
\]
Furthermore,
$B_X$
is a cyclic group of order $2$, $4$ or $6$,
containing $\{ \pm 1\}$
and the case $B_X \simeq \Z/4\Z$ (resp. the case
$B_X \simeq \Z/6\Z$) can occur only if $X$
is an isotrivial elliptic fibration with $j$-invariant
$j = 1728$ (resp. $j = 0$).

Finally, if $X$ is $T$-general, then
$B_X = \widetilde{B}_X =\{\pm 1\}$, that is in this
case $\Jac^k(X)$ and $\Jac^\ell(X)$ are isomorphic
over $\PP^1$ if and only if they are isomorphic
as elliptic surfaces if and only if $k \equiv \pm \ell \pmod{t}$.
\end{proposition}

In the statement
we excluded 
the trivial cases $t = 1, 2$ because
such
elliptic K3 surfaces do not admit non-trivial coprime Jacobians.

Before we give the proof of the proposition, we need to set up some notation.
Let $S$ be an elliptic K3 with a section.
For any subgroup $H \subset A(S,F)$
and any class $\alpha \in \Br(S)$
let $H^{\alpha}$ be the subgroup of $H$ consisting
of elements $\beta \in H$ with the property
$\beta_*(\langle \alpha \rangle) \subset \langle \alpha \rangle$.
Considering the action of $H^\alpha$ 
on $\langle \alpha \rangle = \Z/t\Z$
we get a natural homomorphism $H^\alpha \to (\Z/t\Z)^*$
and we define
\[
\overline{H}^{\alpha} \coloneqq  \Image(H^{\alpha} \to (\Z/t\Z)^*).
\]

\begin{proof}[Proof of Proposition \ref{prop:jacobians-isom}]
Write $S = \Jac^0(X)$. We consider the following subgroups
of $(\Z/t\Z)^*$:
\begin{equation}\label{eq:Bdef}
B_X \coloneqq  \overline{A_{\PP^1}(S)}^{\alpha_X} 
\end{equation}
\begin{equation}\label{eq:tBdef}
\wt{B}_X \coloneqq  \overline{A(S, F)}^{\alpha_X}.
\end{equation}
We have $B_X \subset \wt{B}_X$,
and $-1 \in A_{\PP^1}(S)$ induces
$-1 \in (\Z/t\Z)^*$, in particular $\{ \pm 1\} \subset B_X$. Note that we are assuming $t>2$, hence $-1\not\equiv 1 \pmod{t}$.

By Corollary \ref{cor:Jk-funct}, $\Jac^k(X)$ and $\Jac^\ell(X)$ are isomorphic
over $\PP^1$ if and only if $\Jac^{\ell^{-1}}(\Jac^{k}(X))$ and $\Jac^{\ell^{-1}}(\Jac^{\ell}(X))$ are isomorphic
over $\PP^1$. Here $\ell^{-1}$ is any integer such that $\ell\ell^{-1}\equiv 1\pmod t$. By Lemma \ref{lem:JkJl}, we have $\Jac^{\ell^{-1}}(\Jac^{k}(X))\simeq \Jac^{k\ell^{-1}}(X)$ and $\Jac^{\ell^{-1}}(\Jac^{\ell}(X))\simeq \Jac^{\ell\ell^{-1}}(X)\simeq \Jac^1(X)$ over $\PP^1$, where the last isomorphism follows from \eqref{eq:standardjacobianisomorphisms}. By Corollary \ref{cor:orbitsareisomclasses}, this occurs if and only if $k \ell^{-1} \in B_X$.
By the same argument, $\Jac^k(X)$ and $\Jac^\ell(X)$ are isomorphic
as elliptic surfaces if and only if $\Jac^{k \ell^{-1}}(X)$ and $X$ are isomorphic
as elliptic surfaces if and only if $k \ell^{-1} \in \wt{B}_X$.
The group $B_X$ is a quotient of
a subgroup of $A(S, F)$.
The latter group, by Remark \ref{rem:Pic0-generic}, is isomorphic to the group of group automorphisms of the generic fibre of $S$. Thus, $A(S, F)$ (and hence $B$)
is isomorphic to 
$\Z/2\Z$, unless the
$j$-invariant equals
$1728$ or $0$ in which case $A(S, F)$ (and hence $B_X$) can be $\Z/4\Z$ or $\Z/6\Z$
respectively.

It remains to prove that $B_X = \wt{B}_X = \{\pm 1\}$
if $X$ is $T$-general.
By Proposition \ref{prop:functorialitytorsors}, an isomorphism $X \simeq \Jac^k(X)$ as elliptic surfaces would induce a group automorphism $\beta$ of $S = \Jac^0(X)$ satisfying $\beta_*\alpha_X=k\cdot\alpha_X.$ 
This means that $T(S)$ admits a Hodge isometry $\sigma$, which maps $T(X) = \Ker(\alpha_X)$ to itself.
By $T$-generality, we get $\sigma = \pm \operatorname{id}$ so that
$\beta_* = \pm 1$
and hence $k \equiv \pm 1 \pmod{t}$.
\end{proof}

\begin{corollary} \label{cor:generalcountisequal}
If $A(\Jac^0(X), F) = A_{\PP^1}(\Jac^0(X))$
then isomorphism classes
of coprime Jacobians over $\PP^1$ are the same as isomorphism
classes of coprime Jacobians as elliptic surfaces.
\end{corollary}
\begin{proof}
This follows from Proposition \ref{prop:jacobians-isom}
as in this case
$B_X = \wt{B}_X$ by construction.
\end{proof}

Corollary \ref{cor:generalcountisequal} 
applies when singular fibres of $X \to \PP^1$
lie over a non-symmetric set of points $Z\subset \PP^1$,
that is when
$\ol{\beta} \in \Aut(\PP^1)$ 
satisfies $\ol{\beta}(Z) = Z$
only for $\ol{\beta} = \mathrm{id}$.
On the other hand, if $Z$ is symmetric,
and this symmetry can be lifted to an automorphism
of $\Jac^0(X)$, we typically have $B_X \subsetneq \wt{B}_X$.
For an explicit such surface, see 
Example \ref{ex:fewerasellipticthanoverproj}.

\subsection{$j$-special isotrivial elliptic K3 surfaces}

By a $j$-special isotrivial elliptic K3 surface
we mean an elliptic K3 surface
with smooth fibres all 
having $j$-invariant $0$
or $1728$.
\begin{remark}\label{rem:j0}
There exist Picard rank 2 isotrivial K3 surfaces with $j=0$ (see
Example \ref{ex:thechosenone}), however for $j = 1728$ the minimal rank is $10$ for the following reason.
Let $X$ be an isotrivial 
elliptic K3 surface $X$ with $j = 1728$.
The zeroth Jacobian $S$ of $X$ will have a Weierstrass
equation 
$y^2 = x^3 + F(t) x$ with $F(t)$ a degree 8 polynomial in $t$.
We have $\rho(S) = \rho(X)$.
By semicontinuity of the Picard rank we may assume that
$F(t)$ has distinct roots. In this case $S$ has eight
singular fibres, and the Weierstrass equation 
has ordinary double points
at the singularities of the fibre, so 
$S$ is the result of blowing up 
the Weierstrass model at these $8$ points.
Thus, in addition to the fibre class and the section class,
$S$ has $8$ reducible fibres, so $\rho(S) \ge 10$.
Isotrivial K3
surfaces with $j \ne 0, 1728$
are all Kummer
and hence have Picard rank at least $17$ \cite[Corollary 2]{Saw14}.
\end{remark}

We do not claim a direct relationship between the concepts of
$j$-special and $T$-special, 
however
both of these concepts require extra
automorphisms. 

Let $X\to \PP^1$ be an elliptic K3 surface of multisection index $t > 2$, and let $S\to \PP^1$ be its zeroth Jacobian. Let $H = A_{\PP^1}(S)$;
this group is $\Z/2\Z$ unless
$S$ is $j$-special, in which case
it can be equal to $\Z/4\Z$ (resp. $\Z/6\Z$)
when $j = 1728$ (resp. $j = 0$).
By Proposition
\ref{prop:jacobians-isom}
the number
of coprime Jacobians
of $X$ up to isomorphism
over $\PP^1$ equals
$\phi(t)/|{B_X}|$,
which is
\begin{equation}\label{eq:jspec-formula}
\begin{cases}
\phi(t)/2 &\text{if $X\to \PP^1$ is not isotrivial with $j=0$ or $j=1728$;}\\
\phi(t)/4 &\text{for some isotrivial $X$
with $j=1728$, and $H = \Z/4\Z$;}\\
\phi(t)/6 &\text{for some isotrivial $X$
with $j=0$ and $H = \Z/6\Z$.}
\end{cases}
\end{equation}

We now show that the last two cases are indeed
possible. For simplicity we assume that $t = p$,
an odd prime.

\begin{proposition}\label{prop:specialsurfaces}
Let $S\to \PP^1$ be an elliptic K3 surface with a section. 
Assume $S$ is isotrivial with $j=1728$ (resp. $j = 0$)
and $H = \Z/4\Z$ (resp. $H = \Z/6\Z$).
Let $p > 2$ be a prime.
Then $S$ admits a torsor $X \to \PP^1$
of multisection index $p$ with exactly $\frac{\phi(p)}{4}$
(resp. $\frac{\phi(p)}{6}$)
coprime Jacobians up to isomorphism
over $\PP^1$
if and only if
$p\equiv 1\pmod 4$ 
(resp. $p\equiv 1\pmod 3$).
\end{proposition}
\begin{proof}
Existence of such a torsor $X$ implies the required
numerical
condition on $p$ since $4$ (resp. $6$) divides $\phi(p) = p-1$.

Conversely, assume that
$p$ satisfies the numerical condition.
For every non-trivial element
$\beta \in H$, the fixed subspace
$(T(S) \otimes \CC)^{\langle \beta \rangle}$ is zero; this is because
$S/\langle \beta \rangle$ admits a birational $\PP^1$-fibration over $\PP^1$, hence must be a rational surface.
Thus $T(S) \otimes \CC$, considered
as a representation of a cyclic group $H$
is a direct sum of one-dimensional
representations corresponding
to primitive roots of unity of order $|H|$.

This allows to describe $T(S) \otimes \Q$
as an $H$-representation, because irreducible
$\Q$-representations of $H$ are direct sums
of Galois conjugate one-dimensional representations.
Thus in both cases $T(S) \otimes \Q = V^{\oplus \left(\frac{22-\rho}{2}\right)}$,
where $V$ is the $2$-dimensional
representation $\Q[i] = \Q[x]/(x^2 + 1)$
and $\Q[\omega] = \Q[x]/(x^2 + x + 1)$
respectively.
At this point it follows that under
our assumptions the Picard number
$\rho = \rho(X)$ is even.

On the other hand, 
decomposition of the $H$-representation $T(S) \otimes \Q$
is induced from decomposition
of $T(S) \otimes \Z[1/|H|]$,
hence since $|H|$ is coprime to $p$,
it induces a decomposition 
$T(X) \otimes \F_p \simeq V_p^{\oplus \left(\frac{22-\rho}{2}\right)}$
with $V_p$ defined by
$\F_p[x]/(x^2 + 1)$
and $\F_p[x]/(x^2 + x + 1)$
respectively.
Under the numerical condition on $p$,
the corresponding polynomial has roots and
the representation $V_p$
is a direct sum of
two one-dimensional representations
$V_p = \chi \oplus \chi'$.

It follows that the dual
representation
$\operatorname{Br}(S)_{p-tors}$ 
\eqref{eq:Br-tors}
splits into $1$-dimensional representations
$\chi$, $\chi'$ as well. Take a generator
$\alpha \in \Br(S)_{p-tors}$ for 
one of these representations, 
and let $X$ be the corresponding torsor.
The explicit description 
\eqref{eq:Bdef} shows that $B_X = H$.
\end{proof}

For explicit examples of surfaces 
satisfying conditions of Proposition
\ref{prop:specialsurfaces},
see Example \ref{ex:thechosenone}
and Remark \ref{rem:j0}.
Finally we illustrate the difference
between isomorphism over $\PP^1$
and isomorphism as elliptic surfaces.

\begin{example} \label{ex:fewerasellipticthanoverproj}
Consider the $j=0$
isotrivial
elliptic K3 surface $S\to \PP^1$ of Example \ref{ex:thechosenone}, 
and let $\beta \in A(S,F)$ 
be an automorphism of order $11$. 
Note that $\beta \not\in A_{\PP^1}(S)$
so we may have $B_X \subsetneq \wt{B}_X$
in Proposition \ref{prop:jacobians-isom}.
By Lemma \ref{lem:autgroupwithsection}, $\beta$ acts nontrivially on $T(S)$.
As in the proof of Proposition \ref{prop:specialsurfaces},
we deduce that 
for
every prime $p \equiv 1 \pmod{11}$,
the number of coprime Jacobians up to isomorphism as elliptic surfaces
for an eigenvector torsor
will be $11$ times less
than when they are considered up to isomorphism over
$\PP^1$.
\end{example}

\section{Derived equivalent K3s and Jacobians}
\label{sec:DEFM}

The following well-known result goes back to Mukai,
see also {\cite[Remark 5.4.6]{Cal00}}. 
We provide the proof for completeness
as it follows easily from what we have explained so far.

\begin{theorem}
\label{thm:coprimejacfm}
Let $S\to \PP^1$ be an elliptic K3 surface with a section, and let $X\to\PP^1$ be a torsor over $S\to \PP^1$. Let $t\in \Z$ be the multisection index of $X\to \PP^1$. Then $\Jac^k(X)$ is a Fourier-Mukai partner of $X$ if and only if $\operatorname{gcd}(k,t) = 1$.
\end{theorem}
\begin{proof}
Let $\alpha_X\in \operatorname{Br}(S)$ be the Brauer class of $X\to \PP^1$.
From Lemma \ref{lem:Br-Sha-WC}
it is easy to
deduce that
\begin{equation}\label{eq:TXS-det}
\det(T(X)) = t^2 \cdot \det(T(S))    
\end{equation}
(cf \cite[Remark 3.1]{HS05}).

Recall that $t = \ord(\alpha_X)$
by Lemma \ref{lem:Br-Sha-WC}.
We know $T(\Jac^k(X))$ is Hodge isometric to the kernel of $k\cdot\alpha_X: T(S)\to \Z/t\Z$, again by Lemma \ref{lem:Br-Sha-WC}. If $\operatorname{gcd}(k,t) = 1$, then 
$\alpha_X$ and $k\alpha_X$ have the same kernel so that
\[
T(\Jac^k(X)) \simeq \operatorname{ker}(k\cdot \alpha_X) = \operatorname{ker}(\alpha_X) \simeq T(X),
\]
so $\Jac^k(X)$ is a Fourier-Mukai partner of $X$
by the Derived Torelli Theorem.

Let us prove the
converse implication.
From \eqref{eq:TXS-det}, 
we get
that for any $k\in \Z$, we have \[
\frac{\det(T(X))}{\det(T(\Jac^k(X)))} = \left(\frac{\ord(\alpha)}{\ord(k\alpha)}\right)^2 = \gcd(k,\ord(\alpha))^2.
\]
Thus if $X$ and $\Jac^k(X)$
are derived equivalent, then 
the left-hand side equals one by the Derived Torelli Theorem, hence 
$k$ is coprime to $t = \ord(\alpha)$.
\end{proof}

\subsection{Derived elliptic structures}

In this subsection, we set up the theory of derived elliptic structures and Hodge elliptic structures.

\begin{definition}
Let $X$ be a K3 surface.
A \textit{derived elliptic structure} on $X$ is a pair $(Y,\phi)$, where $Y$ is a K3 surface such that $Y$ is derived equivalent to $X$ and $\phi: Y\to\PP^1$ is an elliptic fibration. 
\end{definition}

 We say that two derived elliptic structures are isomorphic if they are isomorphic as elliptic surfaces.
We denote by $\DE(X)$ (resp. $\DE_t(X)$) the set of isomorphism classes of derived elliptic structures on $X$ (resp. derived elliptic structures on $X$ of multisection index $t$).

\begin{lemma}\label{lem:DE-properties}
Let $X$ be a K3 surface. Then we have:

\begin{enumerate}
    \item[(i)] $\DE(X)$ is a finite set;
    
    \item[(ii)] $\DE(X)$ is nonempty if and only if $X$ is elliptic;
    
    \item[(iii)] $\DE_t(X)$ can be nonempty only for
    $t$ such that $t^2$ divides the order
    of the discriminant group $A_{T(X)}$;
    
    \item[(iv)] If $X$ is elliptic
    with $\rho(X) = 2$ and
    multisection index $t$, then every elliptic structure on every Fourier-Mukai partner of $X$
    also has multisection index $t$, that is
    $\DE(X) = \DE_t(X)$.
\end{enumerate}
\end{lemma}
\begin{proof}
(i) The set of isomorphism classes
of Fourier-Mukai partners
of $X$ is finite \cite[Proposition 5.3]{BM01}, \cite{HLOY02}, 
and each of them
has only finitely many
elliptic structures
up to isomorphism \cite{Ste85}.
It follows that $\DE(X)$ is a finite set.

(ii) If $X$ elliptic, then $X$ with its elliptic structure is an element of $\DE(X)$, hence it is nonempty.
Conversely, if
$\DE(X)$ is nonempty,
then $X$ admits
a Fourier-Mukai partner which is an elliptic K3
surface. Then by the Derived Torelli Theorem
$\NS(X)$ and $\NS(Y)$
are in the same genus, and since $Y$ is elliptic,
the intersection form $\NS(Y)$ 
represents zero, hence a
standard lattice theoretic argument
shows that $\NS(X)$ also represents zero,
and $X$ is elliptic.

(iii)
If $(Y,\phi)$ is a derived elliptic structure of $X$ of multisection index $t$, then we 
have 
\[
|A_{T(X)}| = |A_{T(Y)}| = t^2\cdot |A_{T(\Jac^0(Y))}|
\]
where the first equality follows from the Derived Torelli Theorem and the second one 
can be deduced from \eqref{eq:TXS-det}
(cf \cite[Remark 3.1]{HS05}). 
In particular, $\DE_t(X)$ is empty whenever $t^2$ does not divide the order of $A_{T(X)}$.

(iv) Every 
Fourier-Mukai partner $Y$ of $X$
also has Picard number $\rho(Y) = 2$.
By Proposition \ref{prop:neronseveriranktwo}, the multisection index of every elliptic fibration on $Y$
equals the square root of $|A_{T(Y)}| = |A_{T(X)}| = t^2$.
\end{proof}

We can take coprime Jacobians of 
a derived elliptic structure $(Y,\phi)$, 
which we denote by $\Jac^k(Y,\phi)$. By Lemma \ref{lem:JkJl} and Theorem \ref{thm:coprimejacfm}
this defines a group action of $(\Z/t\Z)^*$ on $\DE_t(X)$.
The set of $(\Z/t\Z)^*$-orbits on
$\DE_t(X)$  parametrizes derived elliptic structures up to taking coprime
Jacobians, and it is sometimes a more natural set to work with.

We now explain Hodge-theoretic analogues of derived elliptic structures. The following definition is motivated by the Derived Torelli Theorem.

\begin{definition}
Let $X$ be a K3 surface.
A \textit{Hodge elliptic structure}
on $X$ is a twisted Jacobian
K3 surface $(S, f, \alpha)$
\textup(see Definition
\ref{def:twisted-jacobian-K3}\textup)
such that there exists a Hodge isometry
$\Ker(\alpha)\simeq T(X)$.
\end{definition}

The index of a Hodge elliptic structure
is defined
to be
the order of its Brauer class $\alpha$.
An isomorphism of Hodge elliptic structures
$(S, f, \alpha)$, $(S', f', \alpha')$
is an isomorphism $\gamma: S \to S'$
of elliptic surfaces
such that $\gamma_*(\alpha) = \alpha'$.
We denote by $\HE(X)$ the set of isomorphism classes of Hodge elliptic structures on $X$.
We write $\HE_t(X)$ for the set of isomorphism classes
of Hodge elliptic structures of index $t$.
The operation $k*(S,f,\alpha) = (S,f,k\alpha)$ defines a group action of $(\Z/t\Z)^*$ on $\HE_t(X)$.

\begin{example}
Let $X$ be an elliptic K3 surface
of Picard rank $2$ and multisection index $t$.
Let $(S, f, \alpha)$ be a Hodge elliptic structure on $X$. Since the discriminant of $X$ equals $t^2$,
from the sequence
\[
0 \to T(X) \to T(S) \to \Z/t\Z \to 0,
\]
we deduce that $T(S)$ is unimodular.
Thus $S$ is an elliptic K3 surface
of Picard rank two,
and it has a unique elliptic fibration, which has a unique section (see Lemma \ref{lem:twofibrations-numerical}).
We see that in the Picard rank two case
$f$ can be excluded from the data of a Hodge elliptic structure and
we have
a bijection
\begin{equation}\label{eq:HE-rk2}
\HE_t(X) = \{ (S, \alpha) \}/ \simeq,
\end{equation}
with isomorphisms understood as isomorphisms between K3 surfaces respecting the Brauer classes.
\end{example}

\begin{proposition}\label{prop:det--brt}
Let $X$ be a K3 surface
and let $t$ be a positive integer. 
Then the bijection $\EllK3 \simeq \BrK3$ of Theorem \ref{thm:ell--br} induces a $(\Z/t\Z)^*$-equivariant bijection $\DE_t(X)\simeq \HE_t(X)$.
\end{proposition}
\begin{proof}
First of all note that by definition $\DE_t(X)$
is a subset of $\EllK3$
consisting of isomorphism classes $(Y, \phi)$
with $Y$ derived equivalent to $X$ and $\phi$
having a multisection index $t$.
Similarly, $\HE_t(X)$ is a subset of $\BrK3$
consisting of $(S, f, \alpha)$
such that $\ord(\alpha) = t$ and $\Ker(\alpha) \simeq T(X)$.
If $(Y, \phi) \in \EllK3$, then
by Lemma \ref{lem:Br-Sha-WC},
$(Y, \phi)$ belongs to $\DE_t(X)$
if and only if the
corresponding triple $(\Jac^0(Y), \Jac^0(\phi), \alpha_Y) \in \BrK3$ 
belongs to $\HE_t(X)$.

The $(\Z/t\Z)^*$-equivariance of the map is a direct consequence of the fact that $k\alpha_Y = \alpha_{\Jac^k(Y)}$, which holds again
by Lemma \ref{lem:Br-Sha-WC}.
\end{proof}

\begin{definition} Let $T$ be a lattice.
For $t\in \Z$, we write $\II_t(A_{T})$ for the set of cyclic, isotropic subgroups of order $t$ in $A_{T},$ and we write $\tII_t(A_T)$ for the set of isotropic vectors of order $t$ in $A_{T}.$
\end{definition}
For a K3 surface $X$, 
there is a natural action of $G_X$, on
$\II_t(A_{T(X)})$ and $\tII_t(A_{T(X)})$. 
Let $(S,f,\alpha)$ be  
a Hodge elliptic structure on $X$
of index $t$. 
There is a unique isomorphism $r_\alpha:\Z/t\Z\simeq T(S)/\Ker(\alpha)$ such that the diagram 
\begin{equation}\label{eq:singledoutgenerator}
    \xymatrix@C=8pt@R=17pt{
    &T(S) \ar[dr] \ar[dl]_{\alpha} & \\
    \Z/t\Z \ar[rr]_-{r_\alpha}^-\sim && T(S)/\Ker(\alpha)
    }
\end{equation}
commutes. In particular, the Brauer class $\alpha$ singles out a generator $r_\alpha(\ol{1})$ of $T(S)/\Ker(\alpha)$. Fix any Hodge isometry $T(X) \simeq \Ker(\alpha)$. 
The natural inclusion $T(S)/T(X)\subset A_{T(X)}$ allows us to view $r_\alpha(\ol{1})$ as an element of $A_{T(X)}$, which we denote by $w_\alpha$. We denote the subgroup of $A_{T(X)}$ generated by $w_\alpha$ by $H_\alpha$. Note that $w_\alpha$, and hence $H_\alpha$, is only well-defined up to the $G_X$ action on $A_{T(X)}$, since its construction depends on the original
choice of Hodge isometry $T(X)
\simeq \Ker(\alpha)$. On the other hand isomorphic Hodge elliptic structures
on $X$
give rise to 
isotropic vectors in the same $G_X$-orbit by Lemma \ref{lem:overlatsubgrcorresp}.
We define the map
\begin{equation}\label{eq:a-map}
w:\HE_t(X) \to \tII_t(A_{T(X)})/G_X,
\quad w(S,f,\alpha) = w_\alpha.
\end{equation}

The operation $k*w = k^{-1}w$, where $k^{-1}$ is an inverse to $k$ modulo $t$, defines a group action of $(\Z/t\Z)^*$ on $\tII_t(A_T)/G_T$.

\begin{lemma}\label{lem:brt--ita}
The map \eqref{eq:a-map}
is $(\Z/t\Z)^*$-equivariant. 
\end{lemma}
\begin{proof}

Recall from Lemma \ref{lem:Br-Sha-WC}(ii) that $\alpha_{\Jac^k(Y)} = k\cdot \alpha_Y$ in $\Br(\Jac^0(Y))$ for all $k\in \Z$. It follows from 
\eqref{eq:singledoutgenerator}
that we have 
$r_{k\alpha} = k^{-1}r_\alpha$.
Thus from the definitions we get
\[
w_{k\alpha} = r_{k\alpha}(\ol{1}) = 
k^{-1}r_{\alpha}(\ol{1}) = 
k^{-1}w_\alpha = k*w_{\alpha},
\]
which means
that the map $w$ is equivariant. 
\end{proof}

Proposition \ref{prop:det--brt}
and Lemma \ref{lem:brt--ita}
give rise to the following commutative diagram
with the vertical arrows being quotients by the corresponding
$(\Z/t\Z)^*$-actions:

\begin{equation}\label{eq:main-maps}
    \xymatrix{
    \DE_t(X) \ar[r]^-\sim \ar[d] & \HE_t(X) \ar[d] \ar[r]^-w & \tII_t(A_{T(X)})/G_X \ar[d]\\
    \DE_t(X) / (\Z/t\Z)^* \ar[r]^\sim
    & \HE_t(X) / (\Z/t\Z)^* \ar[r] & \II_t(A_{T(X)})/G_X
    }
\end{equation} 

For $(Y,\phi)$ a derived elliptic structure of $X$,
we consider $w_\phi
\coloneqq w_{\alpha_Y}$,
the image of $(Y, \phi)$
under the composition of maps
in the top row of \eqref{eq:main-maps}. 
In particular, if $f: X\to \PP^1$ is an elliptic fibration with fibre class $F\in \NS(X)$, then by construction, $w_f$ is the C\u{a}ld\u{a}raru class of the moduli space $\Jac^0(X)$
of sheaves with Mukai vector $(0,F,0)$
on $X$, thus by
Lemma \ref{lem:caldararu-class} 
$w_f$ corresponds to
\begin{equation}\label{eq:lagrangianoffibration}
\frac1{t}F \in \II_t(A_{NS(X)})/G_X
\end{equation}
(we can get rid of the minus sign in the formula at this point,
as $-1 \in G_X$).

\subsection{Fourier-Mukai partners in rank 2}
\label{sec:fmpartnersranktwo}
In this subsection, we work with an elliptic K3 surface of Picard rank 2,
     so that by Proposition \ref{prop:neronseveriranktwo} we have 
     $\NS(X) \simeq \Lambda_{d,t}$ given
     by \eqref{matrix:lambda}.
     The following result is one of the reasons
     why it is natural to concentrate on Picard rank two elliptic surfaces.

\begin{lemma}\label{lem:ranktwo--sameindex}
For an elliptic K3 surface $X$ with $\NS(X)\simeq \Lambda_{d,t}$, 
all derived elliptic
structures 
and all Hodge elliptic structures
on $X$
have the same index $t$.
\end{lemma}
\begin{proof}
This follows
from Lemma
\ref{lem:DE-properties} and Proposition \ref{prop:det--brt}.
\end{proof}

For $X$ as in Lemma \ref{lem:ranktwo--sameindex}, we have $\DE(X) = \DE_t(X)$.
In particular, there is an action of $(\Z/t\Z)^*$ on $\DE(X)$ by taking coprime Jacobians. 
Recall that for a K3 surface $X$ with $\NS(X)\simeq \Lambda_{d,t}$, we have $A_{T(X)}\simeq A_{\NS(X)}(-1)\simeq A_{d,t}(-1)$, 
and it has order $t^2$ by Lemma \ref{lem:lambdabasics}. 
Thus isotropic elements
(resp. 
cyclic isotropic subgroups)
of order $t$
are precisely
Lagrangian elements
(resp. Lagrangian subgroups),
see Definition \ref{def:lagrangianldt}:
\[
\II_t(A_{T(X)}) = \Lagr(A_{T(X)}), \quad 
\tII_t(A_{T(X)}) = \tLagr(A_{T(X)}), \]

The following result is related to \cite[Proposition 3.3]{Ma10}.

\begin{theorem}\label{thm:derivedelliptic-lagrangianelts}
Let $X$ be an elliptic K3 surface of Picard rank 2
and multisection index $t$.
Then the map $w$ \eqref{eq:a-map} is a bijection.
Furthermore, we
have a bijection
\begin{equation}\label{eq:main-bijections}
\DE(X)/(\Z/t\Z)^* \simeq  \Lagr(A_{T(X)})/G_X.
\end{equation}
Action \eqref{eq:involution}
induces a $\Z/2\Z$-action on $\Lagr(A_{T(X)})/G_X$
which under bijection \eqref{eq:main-bijections}
corresponds to the action on $\DE(X)$
swapping 
the two elliptic fibrations
on Fourier-Mukai partners of $X$.
\end{theorem}

\begin{proof}
We first show that $w$ is bijective. 
We start with bijection \eqref{eq:HE-rk2}.
For the injectivity of $w$, take
$(S,\alpha)$ and $(S',\alpha')$
with $T(X) \simeq \Ker(\alpha) \simeq \Ker(\alpha')$.
Assume
that there exists a Hodge isometry $\sigma\in G_X$ with the property 
$\ol{\sigma}(w_\alpha) = w_{\alpha'}$. Then Lemma \ref{lem:overlatsubgrcorresp} implies
 that $\sigma$ can be extended to a Hodge isometry $T(S)\to T(S')$. Since $S$ and $S'$ have Picard rank 2, Lemma \ref{lem:autgroupwithsection} implies that this Hodge isometry is induced by a group isomorphism $\beta: S\simeq S'.$ From $\ol{\sigma}(w_\alpha) = w_{\alpha'}$, it follows 
that $(S,\alpha)$ and $(S',\alpha')$ are isomorphic.

For surjectivity of $w$, let $u \in A_{T(X)}$ 
be an isotropic vector of order $t$
and $H = \langle u\rangle$.
Via Lemma \ref{lem:overlatsubgrcorresp}, $H$ corresponds to an overlattice $i:T(X)\hookrightarrow T$ which inherits a Hodge structure from $T(X)$, i.e. $i:T(X)\hookrightarrow T$ is a Hodge overlattice. Note that $T$ is unimodular, since the index of $T(X)\subset T$ is $t$ and $A_{T(X)}$ has order $t^2$.
Hence $T\oplus U$ is an even, unimodular lattice of rank 22 and signature $(3,19)$. This means that it is isomorphic to the K3-lattice $\Lambda_{\text{K3}}$. By the surjectivity of the period map
(Theorem \ref{thm:surjectivity-period}), 
we obtain a K3 surface $S$ with $T(S)\simeq T$ and $\NS(S) \simeq U$. Therefore the overlattice $i:T(X) \hookrightarrow T(S)$ is a Hodge overlattice with $T(S)/T(X) = H$. We define the Brauer class $\alpha: T(S)\to H\simeq \Z/t\Z$ where the second map is given by $u\mapsto \ol{1}$. Thus we have constructed a pair $(S, \alpha)$ with C\u{a}ld\u{a}raru class $u$ and $\Ker(\alpha) \simeq T(X)$.

Since $w$ is bijective, the 
diagram
\eqref{eq:main-maps}
immediately implies \eqref{eq:main-bijections}.
The action \eqref{eq:involution}
induces the action on $\Lagr(A_{T(X)})/G_X$
because $\iota$ commutes with $G_X$.
Indeed this can be checked on each primary part \eqref{eq:A-primary},
where there are at most two Lagrangian subgroups
(see Lemma \ref{lem:lagrangians-prime}),
hence
the action of $G_X$ factors through 
the action generated by $\iota_p$.
To show that $\iota$ corresponds to swapping
the elliptic fibrations on Fourier-Mukai
partners $Y$, we can use 
the identification
$\Lagr(A_{T(X)})/G_X = \Lagr(A_{T(Y)})/G_Y$,
and assume $Y = X$. The result
follows from \eqref{eq:iota-v-vprime}
because Lagrangian
subgroups generated
by $\ol{v}$ and $\ol{v'}$
correspond to the two elliptic fibrations
on $X$ via \eqref{eq:main-bijections} by \eqref{eq:lagrangianoffibration}.
\end{proof}

Recall from Lemma \ref{lem:twofibrations-numerical} that a K3 surface $X$ with $\operatorname{NS}(X) \simeq \Lambda_{d,t}$ admits two elliptic fibrations, except when $d \equiv -1\pmod t$, in which case $X$ admits only one elliptic fibration. 
Using Theorem \ref{thm:derivedelliptic-lagrangianelts} 
we can easily 
compare the coprime Jacobians of these two fibrations. 

\begin{example}\label{ex:comparing-two-fibrations}
Let $X$  be an elliptic K3 surface
of Picard rank two with $\NS(X) \simeq \Lambda_{d,t}$
such that
$\gcd(d,t)=1$ and
$d \not\equiv -1\pmod t$.
Let $(X,f)$
and
$(X,g)$ be two elliptic fibrations on $X$ (see Lemma \ref{lem:twofibrations-numerical}),
and let $w_f$ and $w_g$ be their C\u{a}ld\u{a}raru classes,
which are Lagrangian elements in $A_{d,t}$.
By Lemma \ref{lem:lagrangians-prime},
$A_{d,t}$ admits a unique Lagrangian subgroup,
thus
we have $\langle w_f \rangle = 
\langle w_g \rangle$.
By 
Theorem \ref{thm:derivedelliptic-lagrangianelts}
this implies that $f$
and $g$ are coprime Jacobians of each other.
We can make this more precise as follows.
Recall that by  \eqref{eq:lagrangianoffibration},
$w_f$ and $w_g$ correspond to classes
$\ol{v}$, $\ol{v'}$ \eqref{eq:v-vprime}
respectively.
Using \eqref{eq:v'-dv}, we compute
$$w_g=\ol{v'} =  
 -d\ol{v}=-d w_f = -d^{-1}*w_f.$$
Here $d^{-1}$ is the inverse to $d$ modulo $t$.
Thus we have an isomorphism of elliptic surfaces
\[
(X,g) \simeq \Jac^{-d^{-1}}(X,f) \simeq \Jac^{d^{-1}}(X,f)
\]
and $(X,f) \simeq \Jac^d(X,g)$.
 \end{example}

\begin{corollary}\label{cor:JFM}
Let $X$ be an elliptic K3 surface of Picard
rank two.
The set of Fourier-Mukai partners of $X$
considered up to isomorphism as surfaces, 
and up to coprime Jacobians (on every derived
elliptic structure of $X$) is in natural bijection with the 
double quotient
\begin{equation*}\label{eq:doublequotient}
\langle\iota\rangle \backslash \Lagr(A_{T(X)})/G_X.
\end{equation*}
\end{corollary}
\begin{proof}
This is the consequence of the action of $\iota$ 
on 
$\Lagr(A_{T(X)})/G_X$ by swapping the two elliptic fibrations 
as explained in Theorem \ref{thm:derivedelliptic-lagrangianelts}.
\end{proof}

\begin{corollary}\label{cor:HT-DJ}
Let $X$ be an elliptic K3 surface of Picard rank 2. Let $d,t\in \Z$ such that $\NS(X)\simeq \Lambda_{d,t}$, and write $m=\operatorname{gcd}(d,t)$. 

\begin{enumerate}
    \item[(i)] If $m=1$, then $\DE(X)$
    is a single $(\Z/t\Z)^*$-orbit. Explicitly, every
    Fourier-Mukai partner of $X$ will be found
    among the coprime
    Jacobians of a fixed elliptic fibration $(X,f)$.
    \item[(ii)] If $m=p^k$, for a prime $p$ and $k \ge 1$,
    then $\DE(X)$ consists of at most two $(\Z/t\Z)^*$-orbits,
    permuted by the involution $\iota$.
    Explicitly every Fourier-Mukai partner of $X$
    will be found among the coprime Jacobians
    of one of the two elliptic fibrations on $X$.
    \item[(iii)] If $m$ has at least $7$ distinct
    prime factors then $\DE(X)$ has at least three
    $(\Z/t\Z)^*$-orbits.
    In particular, there exist Fourier-Mukai partners of $X$ which are not isomorphic, as surfaces, to any of the Jacobians of elliptic structures on $X$.
\end{enumerate}
\end{corollary}
\begin{proof}
In each case we use Theorem 
\ref{thm:derivedelliptic-lagrangianelts}
combined with the count of Lagrangians
given in Proposition \ref{prop:lagrangiancount}.

(i) Fix an elliptic fibration $f: X\to \PP^1$ and let $H_f\subseteq A_{T(X)}$ be the corresponding Lagrangian subgroup. Since $m=1$, Proposition \ref{prop:lagrangiancount} implies that $H_f \subseteq A_{T(X)}$ is the only Lagrangian subgroup. 
Therefore all derived elliptic structures are of the form $\Jac^k(X)\to \PP^1$ for $k\in \Z$ coprime to $t$ by Theorem \ref{thm:derivedelliptic-lagrangianelts}.

(ii) By Proposition \ref{prop:lagrangiancount}, $A_{T(X)}$ contains precisely two Lagrangian subgroups. 
The condition $m = p^k$ implies
in particular that $d \not\equiv -1 \pmod{t}$, hence
the surface $X$ admits two elliptic fibrations $f:X\to \PP^1$ and $g:X\to \PP^1$. By Lemma \ref{lem:twolagrangians-equal}, 
arguing like in Example \ref{ex:comparing-two-fibrations}, we see that
the subgroups of $A_{T(X)}$ induced by the two elliptic fibrations are not equal.
Hence $H_f$ and $H_g$ are the only two Lagrangians of $A_{T(X)}$, so every derived elliptic structure on $X$ is either a coprime Jacobian of $f$ or of $g$ by Theorem \ref{thm:derivedelliptic-lagrangianelts}.

(iii)
Assume $\omega(m) \ge 7$.
Since $-1\in G_X$ acts trivially on $\Lagr(A_{T(X)})$ and $|G_X|\leq 66$, 
by Proposition \ref{prop:lagrangiancount},
the set $\Lagr(A_{T(X)})/G_X$ 
has cardinality at least
$2^{\omega(m)}/33 \geq 128/33$, that is there are at least three elements.
The final statement follows from Corollary \ref{cor:JFM}.
\end{proof}

\begin{corollary}\label{cor:DE-count}
Assume that
$X$ is a $T$-general elliptic K3 surface
with $\NS(X) = \Lambda_{d,t}$ with $t > 2$,
and let
$m=\gcd(d,t)$.
Then 
\begin{equation}
\label{eq:DE-final}
|\operatorname{DE}(X)| = 2^{\omega(m)-1}\cdot \phi(t),
\quad
|\DE(X)/(\Z/t\Z)^*|=2^{\omega(m)}.
\end{equation}

In particular, if 
$m$ is not a power of a prime, then $X$ has
Fourier-Mukai partners
not isomorphic, as surfaces,
to any Jacobian of an elliptic structure on $X$.
\end{corollary}
\begin{proof}
The second formula
in \eqref{eq:DE-final}
is an immediate
consequence
of Theorem \ref{thm:derivedelliptic-lagrangianelts},
the fact that $G_X = \{\pm 1\}$
acts trivially on $\tLagr(A_{T(X)})$ 
and the Lagrangian count
\eqref{eq:lagr-count}.

By Proposition \ref{prop:jacobians-isom}, coprime 
Jacobians of a $T$-general elliptic 
K3 surface form 
$\phi(t)/2$
isomorphism classes.
In other words, the orbits
of the $(\Z/t\Z)^*$-action
on $\DE(X)$
are all of size $\phi(t)/2$
and the first
formula
in \eqref{eq:DE-final}
follows from the second one.

The final statement follows from Corollary
\ref{cor:JFM}
because if $m$ is not a power of a prime, $\DE(X)/(\Z/t\Z)^*$
has at least four elements by \eqref{eq:DE-final}
which thus
can not form a single $\iota$-orbit.
\end{proof}

\subsection{The zeroth Jacobian}
In this subsection, we apply the results of Section \ref{sec:fmpartnersranktwo} to investigate whether derived equivalent elliptic K3 surfaces have isomorphic zeroth Jacobians. A priori, this is a weaker question than Question \ref{q:HT}. However, we now show that the two questions are equivalent in the very general case. In particular, the answer is negative.
\begin{proposition}\label{prop:isomorphiczerothjacobians}
Let $f: X \to \PP^1$ be an elliptic K3 surface 
 of Picard rank 2, and write $S\coloneqq \Jac^0(X)$. Assume that $T(X)$ has no non-trivial rational Hodge isometries, that is
\begin{equation}\label{eq:hodgeassumption}
O_{\text{Hodge}}(T(X)_\Q) \simeq \Z/2\Z.
\end{equation}
Let $(Y,\phi)$ be a derived elliptic structure on $X$ such that $S'\coloneqq\Jac^0(Y) \simeq S$. Then $(Y,\phi)$ is isomorphic to a coprime Jacobian of $(X,f)$. 
\end{proposition}

\begin{proof}
Fixing any Hodge isometry $T(X)\simeq T(Y)$ we view $T(X)\simeq T(Y)\hookrightarrow T(S')$ as an overlattice of $T(X)$. By assumption there exists a Hodge isometry $\beta^*:T(S')\simeq T(S)$ induced by an isomorphism $\beta: S\simeq S'$. Now $\beta^*$ induces the rational Hodge isometry $$T(X)_\Q\simeq T(S')_\Q\overset{\beta^*_\Q}{\simeq}T(S)_\Q\simeq T(X)_\Q$$ which by assumption equals $\pm \operatorname{id},$ hence $\beta^*$ preserves $T(X)$ as a sublattice of $T(S)$ and $T(S')$. In particular, $\beta_*\alpha_X = k\alpha_Y$ for some $k\in \Z$, hence $Y$ is a coprime Jacobian of $X$.
\end{proof}

It is well-known that 
if $X$ 
is
a very general 
$\Lambda_{d,t}$-polarized elliptic K3 surface
then \eqref{eq:hodgeassumption}
is satisfied, see e.g. the argument of
\cite[Lemma 3.9]{SZ20}.
Thus, if $X$ is a very general elliptic K3 surface of Picard rank two with two elliptic fibrations, Proposition \ref{prop:isomorphiczerothjacobians}  allows us to compare the corresponding zeroth Jacobians, which generalises \cite[Proposition 4.8]{Gee05}.

\begin{corollary}\label{cor:Jac0-different}
    Let $X$ be an elliptic K3 surface of Picard rank two with $\NS(X)\simeq \Lambda_{d,t}$ and suppose 
    $d\not\equiv\pm1\mod t,$ so that $X$ admits two non-isomorphic elliptic fibrations by Lemma \ref{lem:twofibrations-numerical}. Assume \eqref{eq:hodgeassumption} holds for $X$. Then the zeroth Jacobians of the two elliptic fibrations on $X$ are isomorphic if and only if $\gcd(d,t)=1$.
\end{corollary}
\begin{proof}
    If $\gcd(d,t)=1$, the two fibrations on $X$ are coprime Jacobians of each other by Corollary \ref{cor:HT-DJ}, hence the zeroth Jacobians are isomorphic. If $\gcd(d,t)\neq 1$, then by $T$-generality of $X$, the C\u{a}ld\u{a}raru classes of the two fibrations on $X$ are not proportional in $A_{T(X)}$, hence the two fibrations are not coprime Jacobians of each other and the result follows from Proposition \ref{prop:isomorphiczerothjacobians}.
\end{proof}

\begin{remark}\label{rem:DE-Jac0}
In the setting of
Corollary \ref{cor:Jac0-different},
if zeroth Jacobians
are not isomorphic, then they 
are also not derived equivalent.
Indeed, elliptic K3 surfaces with a section do not admit nontrivial Fourier-Mukai partners
\cite[Proposition 2.7(3)]{HLOY02}. 
\end{remark}

\subsection{Question by Hassett and Tschinkel
over non-closed fields} \label{sec:nonclosedfields}

In this subsection, we will use the theory of twisted forms to extend our results to a subfield $k \subset \CC$. 
Let $f:X\to \PP^1$ be a complex elliptic K3 surface with $\operatorname{NS}(X) \simeq \Lambda_{d,t}$. 
Recall that we denote by
$\Aut(X,F)$ the group of automorphisms of $X$ which fix the class of the fibre in $\operatorname{NS}(X)$.
By Corollary \ref{cor:autfixingfibre}, 
the group $\operatorname{Aut}(X,F)$ is trivial whenever $t>2$ and $X$ is $T$-general.

Let $k\subset L$ be a field extension.
An $L$\textit{-twisted form} of an elliptic K3 surface $(Y,\phi: Y \to C)$ over $k$ is any elliptic K3 surface $(Y',\phi': Y' \to C')$ over $k$ such that $(Y_L,\phi_L)$ is isomorphic to $(Y'_L,\phi'_L)$ as elliptic surfaces. 

\begin{lemma}\label{lem:notwistedforms}
Let $(Y,\phi)$ be an elliptic K3 surface over $k$ such that $\Aut(Y_\CC,F) = \left\{\operatorname{id}\right\}.$ 
Then every $\CC$-twisted
form of 
$(Y,\phi)$
is isomorphic to $Y$ as a surface.
\end{lemma}
\begin{proof}
Any $\CC$-twisted form $(Y',\phi')$ of $(Y,\phi)$ is also a $\overline{k}$-twisted form of $(Y,\phi)$ \cite[Lemma 16.27]{Mil08}. Thus
it suffices to show that
for any Galois extension
$L/k$ all
$L$-twisted forms
of
$(Y, \phi)$
are isomorphic to $Y$.
Let $(Y',\phi')$ be an 
$L$-twisted form of $(Y,\phi)$, and let $g:Y_L\simeq Y'_L$ be an isomorphism 
of elliptic surfaces,
possibly twisting the base
by an automorphism.
Then for any $\sigma\in \operatorname{Gal}(L/k)$, the map $h\coloneqq g \circ (\sigma g)^{-1}$ 
is an automorphism of $Y_L$ 
as an elliptic surface.

Using injectivity of the map $\operatorname{Aut}(Y_L)\to \operatorname{Aut}(Y_\CC)$, c.f. \cite[Lemma 02VX]{stacks-project}, and the assumption about automorphisms of $Y_\CC$,
we see that
$h$ is the identity, that is $g$ commutes with the Galois action. 
Therefore $g$ 
descends to an isomorphism $Y \simeq Y'$ \cite[Proposition 16.9]{Mil08}. 
\end{proof}

\begin{lemma}\label{lem:elliptic-fibrations-k}
If $(X,f)$ is an elliptic K3 surface
over $k$ such that $\rho(X_{\CC}) = 2$, then all elliptic fibrations of $X_\CC$ are induced by elliptic fibrations of $X$.
\end{lemma}
\begin{proof}
By Lemma \ref{lem:twofibrations-numerical} $X_\CC$
has one or two elliptic fibrations.
If there is only fibration, it must come from the given elliptic fibration $f$.
If there are two elliptic fibrations on $X_\CC$, they are defined over some Galois extension $L/k$.
Let $F$ and $F'$ be the corresponding 
divisor classes on $X_L$.
These classes can not be permuted by the Galois group, because one of them corresponds to $f$, hence is fixed by the Galois group.
Thus the other class is also fixed by the Galois group and the corresponding morphism $X \to C$ is defined over $k$, see e.g. \cite[Proposition 2.7, Theorem 3.4(2)]{Liedtke}.
\end{proof}

\begin{proposition}\label{prop:hassetttschinkelnonclosed}
    Let $X$ be an elliptic 
    K3 surface over $k$ with $\NS(X_\CC)\simeq \Lambda_{d,t}$.
    Assume $\Aut(X_\CC,F) = \left\{\operatorname{id}\right\}.$
    If $d$ and $t$ are coprime
    or have only one prime factor in common,
    then every Fourier-Mukai partner
    of $X$ is isomorphic, as a surface, to a coprime Jacobian
    of one of the elliptic fibrations on $X$.
\end{proposition}
\begin{proof}
Let $Y$ be a
Fourier-Mukai partner of $X$, and let $\phi:Y\to C$ be an elliptic fibration of $Y$, 
which
exists by \cite[Proposition 16]{HT14}.
By Corollary \ref{cor:HT-DJ}(i, ii),
$\phi_\CC: Y_\CC\to C_\CC$
is isomorphic to a coprime Jacobian $\Jac^k(X_\CC, f_\CC)$ as elliptic surfaces, for some elliptic fibration $f_\CC$ on $X_\CC$.
By Lemma \ref{lem:elliptic-fibrations-k}, $f_\CC$ comes from
an elliptic fibration $f$ on $X$, 
hence $(Y,\phi)$ is a $\CC$-twisted form of $\Jac^k(X, f)$.

From the description of the automorphism groups given in Proposition \ref{prop:autsequence}
we deduce that
\[
\Aut(\Jac^k(X_\CC), F) \simeq \Aut(X_\CC,F) 
\]
and by assumption this group is trivial.
It follows from Lemma \ref{lem:notwistedforms} that
$Y$ is isomorphic to $\Jac^k(X)$ as a surface.
\end{proof}

Proposition 
\ref{prop:hassetttschinkelnonclosed}
implies the following:
\begin{corollary}\label{cor:ratpointfmpartners}
Let $X$ be as in Proposition \ref{prop:hassetttschinkelnonclosed}. 
Let $Y$ be any
Fourier-Mukai partner of $X$.
Then $X$ has a $k$-rational point if and only if $Y$ has a $k$-rational point.
\end{corollary}
\begin{proof}
From Proposition \ref{prop:hassetttschinkelnonclosed}, it follows that there is an elliptic fibration $f:X\to C'$ and an integer $\ell\in \Z$ such that $Y\simeq \Jac^\ell(X,f)$ as surfaces. There is a rational map $X\dashrightarrow \Jac^\ell(X)\simeq Y$ given by $P\mapsto \ell\cdot P$. By the Lang-Nishimura Theorem \cite{Lan54}, \cite{Nis55}, it follows that $X(k)\neq \emptyset$ implies $Y(k)\neq \emptyset$. Conversely, since $X$ is also a coprime Jacobian of $Y$, the same argument shows that $Y(k)\neq\emptyset$ implies $X(k)\neq \emptyset$.
\end{proof}

\printbibliography
\end{document}